\documentclass[letterpaper,11pt,twoside,keywordsasfootnote,addressatend,noinfoline]{article}

\usepackage{fullpage}
\usepackage[english]{babel}
\usepackage{amssymb}
\usepackage{amsmath}
\usepackage{theorem}
\usepackage{epsfig}
\usepackage{subfigure}
\usepackage{imsart}
\usepackage{color}

\theorembodyfont{\slshape}
\newtheorem{theor}{Theorem}
\newtheorem{conjc}[theor]{Conjecture}
\newtheorem{lemma}[theor]{Lemma}
\newtheorem{corol}[theor]{Corollary}

\newenvironment{proof}{\noindent{\scshape Proof.}}{\hspace{2mm} $\square$}

\newcommand{\Z}{\mathbb{Z}}
\newcommand{\D}{\mathbb{D}}
\newcommand{\R}{\mathbb{R}}
\newcommand{\ep}{\epsilon}
\newcommand{\ind}{\mathbf 1}

\DeclareMathOperator{\card}{card}
\DeclareMathOperator{\supp}{supp}

\begin{document}

\begin{frontmatter}

\title     {Stochastic dynamics on hypergraphs and \\ the spatial majority rule model}
\runtitle  {Stochastic dynamics on hypergraphs}
\author    {N. Lanchier\thanks{Research supported in part by NSF Grant DMS-10-05282.} and J. Neufer}
\runauthor {N. Lanchier and J. Neufer}
\address   {School of Mathematical and Statistical Sciences, \\ Arizona State University, \\ Tempe, AZ 85287, USA.}
% \address   {School of Mathematical and \\ Statistical Sciences, \\ Arizona State University, \\ Tempe, AZ 85287, USA.}

\begin{abstract} \ \ This article starts by introducing a new theoretical framework to model spatial systems which is obtained
 from the framework of interacting particle systems by replacing the traditional graphical structure that defines the
 network of interactions with a structure of hypergraph.
 This new perspective is more appropriate to define stochastic spatial processes in which large blocks of vertices may flip
 simultaneously, which is then applied to define a spatial version of the Galam's majority rule model.
 In our spatial model, each vertex of the lattice has one of two possible competing opinions, say opinion 0 and opinion 1,
 as in the popular voter model.
 Hyperedges are updated at rate one, which results in all the vertices in the hyperedge changing simultaneously their
 opinion to the majority opinion of the hyperedge.
 In the case of a tie in hyperedges with even size, a bias is introduced in favor of type 1, which is motivated by
 the principle of social inertia.
 Our analytical results along with simulations and heuristic arguments suggest that, in any spatial dimensions and when
 the set of hyperedges consists of the collection of all $n \times \cdots \times n$ blocks of the lattice, opinion 1 wins
 when $n$ is even while the system clusters when $n$ is odd, which contrasts with results about the voter model in high
 dimensions for which opinions coexist.
 This is fully proved in one dimension while the rest of our analysis focuses on the cases when $n = 2$ and $n = 3$ in
 two dimensions.
\end{abstract}



\end{frontmatter}

%%%%%%%%%%%%%%%%%%%%%%%%%%%%%%%%%%%%%%%%%%%%%%%%%%%%%%%%%%%%%%%%%%%%%%%%%%%%%%%%%%%%%%%%%%%%%%%%%%%%%%%%%%%%%%%%%%%%%%%%%%%%%%%%%%%%%%%%%%

\section{Introduction}
\label{sec:intro}

\indent There has been recently a growing interest in spatial models of social and cultural dynamics, where space is modeled through an
 underlying graph whose vertices represent individuals and edges potential dyadic interactions (see
 Castellano \emph{et al} \cite{castellano_fortunato_loreto_2009} for a review).
 Such systems are commonly called agent-based models while the mathematical term is interacting particle systems.
 There has been an increasing effort of applied scientists to understand such systems based on either heuristic arguments or numerical
 simulations.
 This effort must be complemented by analytical studies of interacting particle systems of interest in social sciences, mainly because
 stochastic spatial simulations are known to be difficult to interpret and might lead to erroneous conclusions.
 However, while the mathematical study of interacting particle systems over the past forty years has successfully led to a better
 understanding of a number of physical and biological systems, much less attention has been paid to the field of social sciences,
 with the notable exception of the popular voter model introduced independently in \cite{clifford_sudbury_1973, holley_liggett_1975}.
 The first objective of this paper is to extend the traditional framework of interacting particle systems by replacing the underlying
 graph that represents the network of interactions with the more general structure of a hypergraph, which is better suited to model social
 dynamics.
 The second objective is to use this new framework to construct a spatial version of the majority rule model proposed by
 Galam \cite{galam_2002} to describe public debates, and also to initiate a rigorous analysis of this model of interest in social sciences.
 Even though our proofs rely mostly on new techniques, our model is somewhat reminiscent of the voter model, so we start with a description
 of the voter model and a review of its main properties. \vspace{-8pt} \\
%  However, while the mathematical study of interacting particle systems over the past forty years has successfully led to a better
%  understanding of a number of physical and biological systems, the field of social dynamics has been essentially ignored by mathematicians,
%  with the notable exception of the voter model introduced in \cite{clifford_sudbury_1973, holley_liggett_1975}.
%  Using traditional techniques in the field of interacting particle systems as well as new techniques, the first author and other researchers
%  have recently initiated this study by collecting important analytical results for two popular models of social dynamics, the Axelrod model
%  and the Deffuant model, that we review below.
%  This article continues in this spirit with the majority rule model proposed by Galam to model public debates. \vspace{-8pt} \\

\noindent {\bf The voter model} --
 In the voter model, individuals are located on the vertex set of a graph and are characterized by one of two possible competing
 opinions, say opinion 0 and opinion 1.
 In particular, the state of the system at time $t$ is a so-called spatial configuration $\eta_t$ that can be viewed either as a
 function that maps the vertex set into $\{0, 1 \}$, in which case the value of $\eta_t (x)$ represents the opinion of vertex $x$ at
 time $t$, or as a subset of the vertex set, i.e., the set of vertices with opinion~1 at time~$t$.
 The system evolves as follows:
 individuals independently update their opinion at rate one, i.e., at the times of independent Poisson processes with intensity one, by
 mimicking one of their nearest neighbors chosen uniformly at random.
 Since the times between consecutive updates are independent exponential random variables, the probability of two simultaneous updates
 is equal to zero, therefore the process is well-defined when the number of individuals is finite.
 On infinite graphs, since there are infinitely many Poisson processes, the time to the first update does not exist, and to prove
 that the voter model is well-defined, one must show in addition that the opinion at any given space-time point results from only finitely
 many interactions.
 This follows from an argument due to Harris \cite{harris_1972}.
 The idea is that for $t > 0$ small the vertex set can be partitioned into finite islands that do not interact with each other by time $t$,
 which allows to construct the process independently on each of those islands up to time $t$, and by induction at any time. \\
\indent We now specialize in the $d$-dimensional regular lattice in which each vertex is connected by an edge to each of its $2d$
 nearest neighbors.
 The argument of Harris \cite{harris_1972} mentioned above allows to construct the voter model graphically from a collection of independent
 Poisson processes, but also to prove a certain duality relationship between the voter model and a system of coalescing random walks
 on the lattice.
 Using this duality relationship and the fact that simple symmetric random walks are recurrent in $d \leq 2$ but transient in $d \geq 3$,
 one can prove that clustering occurs in one and two dimensions, i.e., any two vertices eventually share the same opinion which
 translates into the formation of clusters that keep growing indefinitely, whereas coexistence occurs in higher dimensions, i.e.,
 the process converges to an equilibrium in which any two vertices have a positive probability of having different
 opinions \cite{clifford_sudbury_1973, holley_liggett_1975}. \\
\indent In view of the dichotomy clustering/coexistence depending on the dimension, a natural question about the voter model is:
 how fast do clusters grow in low dimensions?
 In one dimension, Bramson and Griffeath \cite{bramson_griffeath_1980} proved that the size of the clusters scales asymptotically like
 the square root of time:
 assuming that two vertices are distance $t^a$ apart at time $t$, as $t \to \infty$, both vertices have independent opinions
 when $a > 1/2$ whereas they are totally correlated, i.e., the probability that they share the same opinion tends to one, when $a < 1/2$.
 In contrast, there is no natural scale for the cluster size in two dimensions, a result due to Cox and Griffeath \cite{cox_griffeath_1986}.
 More precisely, both vertices again have independent opinions when $a > 1/2$ but the probability that they share the same opinion when
 $a < 1/2$ is no longer equal to one but linear in $a$. \\
\indent In higher dimensions, though coexistence occurs, the local interactions that dictate the dynamics of the voter
 model again induce spatial correlations, and a natural question is: how strong are the spatial correlations at equilibrium?
 To answer this question, the idea is to look at the random number of vertices with opinion 1 at equilibrium in a large cube minus its
 mean in order to obtain a centered random variable.
 Then, according to the central limit theorem, if opinions at different vertices were independent, this random variable rescaled by
 the square root of the size of the cube would be well approximated by a centered Gaussian.
 However, such a convergence is obtained for a larger exponent in the renormalization factor, a result due to Bramson and
 Griffeath \cite{bramson_griffeath_1979} in three dimensions and extended by Z\"ahle \cite{zahle_2001} to higher dimensions.
 This indicates that, even in high dimensions, spatial correlations are still significant. \\
\indent Another natural question about the voter model is: does the opinion of any given vertex stay the same after a finite random time?
 To answer this question, we look at the fraction of time a given vertex is of type 1 in the long run, a random variable called the
 occupation time.
 When coexistence occurs, vertices change their opinion infinitely often therefore it is expected that the occupation time converges
 almost surely to the initial density of type 1.
 In contrast, when clustering occurs, it is expected that any given vertex will eventually be covered by a cluster and that the previous
 law of large number does not hold.
 However, Cox and Griffeath \cite{cox_griffeath_1983} proved almost sure convergence of the occupation time to the initial density of
 type 1 in dimensions $d \geq 2$.
 This does not hold in one dimension but it can be proved that, even in this case, the type of any given vertex keeps changing indefinitely.
 From the combination of the previous results, we obtain the following description of the voter model on the one- and two-dimensional
 lattices.
 Clusters form and appear to grow indefinitely so only one type is present at equilibrium.
 However, any given vertex flips infinitely often, which also indicates that clusters are not fixed in space but move around, and thus
 may give the impression of local transience though, strictly speaking, coexistence does not occur. \vspace{-8pt} \\

\noindent {\bf The majority rule model} --
 The main objective of this article is to initiate a rigorous analysis of another interacting particle system of interest in social
 sciences: a spatial version of the majority rule model proposed by Galam \cite{galam_2002} to describe public debates.
 In the original nonspatial model, the population is finite and each agent is either in state zero or in state one, representing two
 differing opinions.
 At each time step, a positive integer, say $n$, is randomly chosen according to a given distribution, then $n$ agents are chosen
 uniformly at random from the population.
 These agents form a discussion group which results in all $n$ agents changing simultaneously their opinion to the majority opinion
 of the group.
 The majority rule is well-defined when $n$ is odd while, when $n$ is even and a tie occurs, a bias is introduced in favor of one
 opinion, say opinion 1, which is motivated by the principle of social inertia.
 Note that most (if not all) models of interacting particle systems studied in the mathematics literature are naturally defined
 through an underlying graph which encodes the pairs of vertices that may interact using edges.
 However, to define spatial versions of the majority rule model, one needs a more complex network of interactions since the
 dynamics do not reduce to dyadic interactions: vertices interact by blocks.
 The most natural mathematical structure that can model such a network is the structure of hypergraph.
 In order to define spatial versions of the majority rule model, we thus extend the traditional definition of interacting particle
 systems by replacing the underlying graph structure with that of a hypergraph, an approach that we propose as a new modeling
 framework to describe social and cultural dynamics. \vspace{-8pt} \\

\begin{figure}[t!]
\centering
\includegraphics[width=400pt]{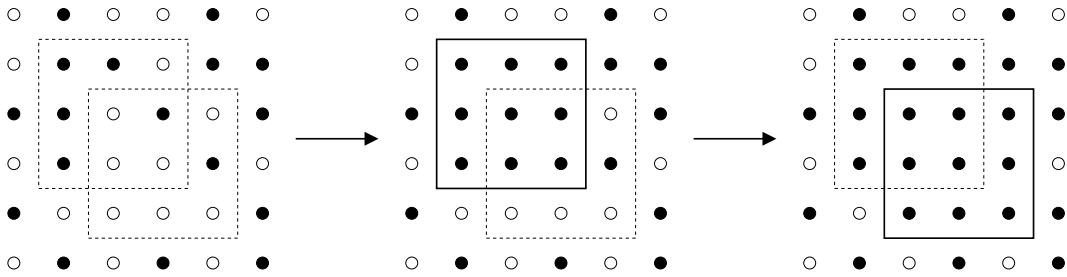}
\caption{\upshape{Dynamics on hypergraph}}
\label{fig:network}
\end{figure}

\pagebreak

\noindent {\bf Stochastic dynamics on hypergraphs} --
 In a number of coordinated sociological systems, people's opinions are subject to change in large groups due to, e.g., the influence
 of an opinion leader or public debates.
 To model such systems, we replace the underlying graph structure of traditional particle systems with one of a hypergraph, which has
 a set of vertices, as does a graph, but the set of edges is replaced by a set of hyperedges, which are no longer limited to a connection
 between two vertices modeling dyadic interactions: hyperedges are nonempty subsets of vertices which can be arbitrarily large.
 In the context of social dynamics, each hyperedge can be thought of as a social group such as a family unit, a team of co-workers,
 or a group of classmates, in which members interact simultaneously.
 To define the framework mathematically, we let $\mathbb H = (V, H)$ be a hypergraph and, for simplicity, restrict ourselves to spin
 systems, i.e., the individual at each vertex is characterized by one of only two possible opinions.
 In particular, as for the voter model described above, the state of the system at time $t$ is a so-called spatial configuration
 $\eta_t$ that can be viewed either as a function that maps the vertex set into $\{0, 1 \}$, in which case the value of $\eta_t (x)$
 represents the opinion of vertex $x$ at time $t$, or as a subset of the vertex set, i.e., the set of vertices with opinion~1 at
 time~$t$.
 Then, we define the dynamics using a Markov generator of the form
\begin{equation}
\label{eq:generator}
  Lf (\eta) \ = \ \sum_{h \in H} \ \sum_{A \subset h} \ c_A (\eta \cap h) \ [f (\eta_{A, h}) - f (\eta)] \quad
      \hbox{where} \ \ \eta_{A, h} = (\eta \setminus h) \cup A.
\end{equation}
 Equation \eqref{eq:generator} simply means that for every hyperedge $h$ and every subset $A \subset h$ of the hyperedge,
 the spins of all vertices in $A$ become 1 while the spins of all the other vertices in the hyperedge become 0, at a rate that
 only depends on $A$ and the configuration in the hyperedge, which is thus denoted by the function $c_A (\eta \cap h)$.
 In particular, we point out that each update of the system corresponds to the simultaneous update of all the vertices in a given
 hyperedge, rather than a single vertex, and that the rate at which a hyperedge is updated only depends on the configuration in this
 hyperedge.
 Also, we recall that an update occurs at rate $c$ if the waiting time for this update is exponentially distributed with mean $1/c$.
%  Note also that the use of hypergraph structures allows a generalization of interacting particle systems since traditional graphs
%  simply are hypergraphs where all hyperedges have cardinality two.
 Returning to the majority rule model, to design a spatial analog that also accounts for the social structure of the population --
 in which discussions occur among agents that indeed belong to a common social group rather than agents chosen at random -- it is
 natural to use the framework of stochastic dynamics on hypergraphs, where each hyperedge represents a social group.
 The dynamics of the majority rule are described by the Markov generator
\begin{equation}
\label{eq:majority}
 \begin{array}{l}
    \displaystyle Lf (\eta) \ = \
    \displaystyle \sum_{h \in H} \ \ind \,\{\card \,(\eta \cap h) < (\card h) / 2 \} \ [f (\eta \setminus h) - f (\eta)] \vspace{-4pt} \\ \hspace{80pt} + \
    \displaystyle \sum_{h \in H} \ \ind \,\{\card \,(\eta \cap h) \geq (\card h) / 2 \} \ [f (\eta \cup h) - f (\eta)] \end{array}
\end{equation}
 where $\ind$ is the indicator function, which is equal to one if its argument is true and zero otherwise, and where $\card$
 stands for the cardinality.
 In particular, $\card \,(\eta \cap h)$ is the number of individuals with opinion 1 in the hyperedge $h$.
 Note that \eqref{eq:majority} simply is a particular case of \eqref{eq:generator} with
 $$ \begin{array}{rcl}
     c_A (\eta \cap h) & = & 0 \quad \hbox{for} \ A \notin \{\varnothing, h \} \vspace{4pt} \\
     c_{\varnothing} (\eta \cap h) & = & \ind \,\{\card \,(\eta \cap h) < (\card h) / 2 \} \vspace{4pt} \\
     c_h (\eta \cap h) & = & \ind \,\{\card \,(\eta \cap h) \geq (\card h) / 2 \}. \end{array} $$
 Note also that the first indicator function is equal to one, and the second one equal to zero, if and only if there is a
 strict majority of type 0 in the hyperedge $h$.
 In particular, the expression of the generator \eqref{eq:majority} means that the configuration in each hyperedge is
 updated at rate one
% , i.e, at the times of a Poisson process with parameter one,
 with all the vertices changing simultaneously their
 opinion to the majority opinion in the hyperedge.
 In case of a tie, opinion 1 is adopted.
 Since hyperedges are updated at the times of independent Poisson processes with intensity one, the probability of two overlapping
 hyperedges being updated simultaneously is equal to zero, therefore the process is well-defined on finite hypergraphs.
 Moreover, the argument of Harris \cite{harris_1972} described above to justify the existence of the voter model on infinite graphs
 also applies to the majority rule model, therefore our process is well-defined on finite hypergraphs as well.
 For simplicity, we assume from now on that the vertex set is $\Z^d$, we let $n > 1$ be a nonrandom integer, and we define the collection of
 hyperedges by
 $$ H \ = \ \{h_x : x \in \Z^d \} \quad \hbox{where} \quad h_x \ = \ x + \{0, 1, 2, \ldots, n - 1 \}^d. $$
 Each social group is thus represented by a $n \times \cdots \times n$ block on the lattice.
 Note that, even though this might look simplistic, any two vertices are connected by paths of overlapping hyperedges, which results
 in a system that exhibits spatial correlations and nontrivial dynamics.
 Figure \ref{fig:network} gives an example of realization when $d = 2$ and $n = 3$ with two consecutive updates.
 Time goes from left to right and at each step the hyperedge which is updated is framed with thick continuous lines.
 Black and white dots refer to opinion 1 and opinion 0, respectively.

\begin{figure}[t!]
\centering
\mbox{\subfigure{\epsfig{figure = 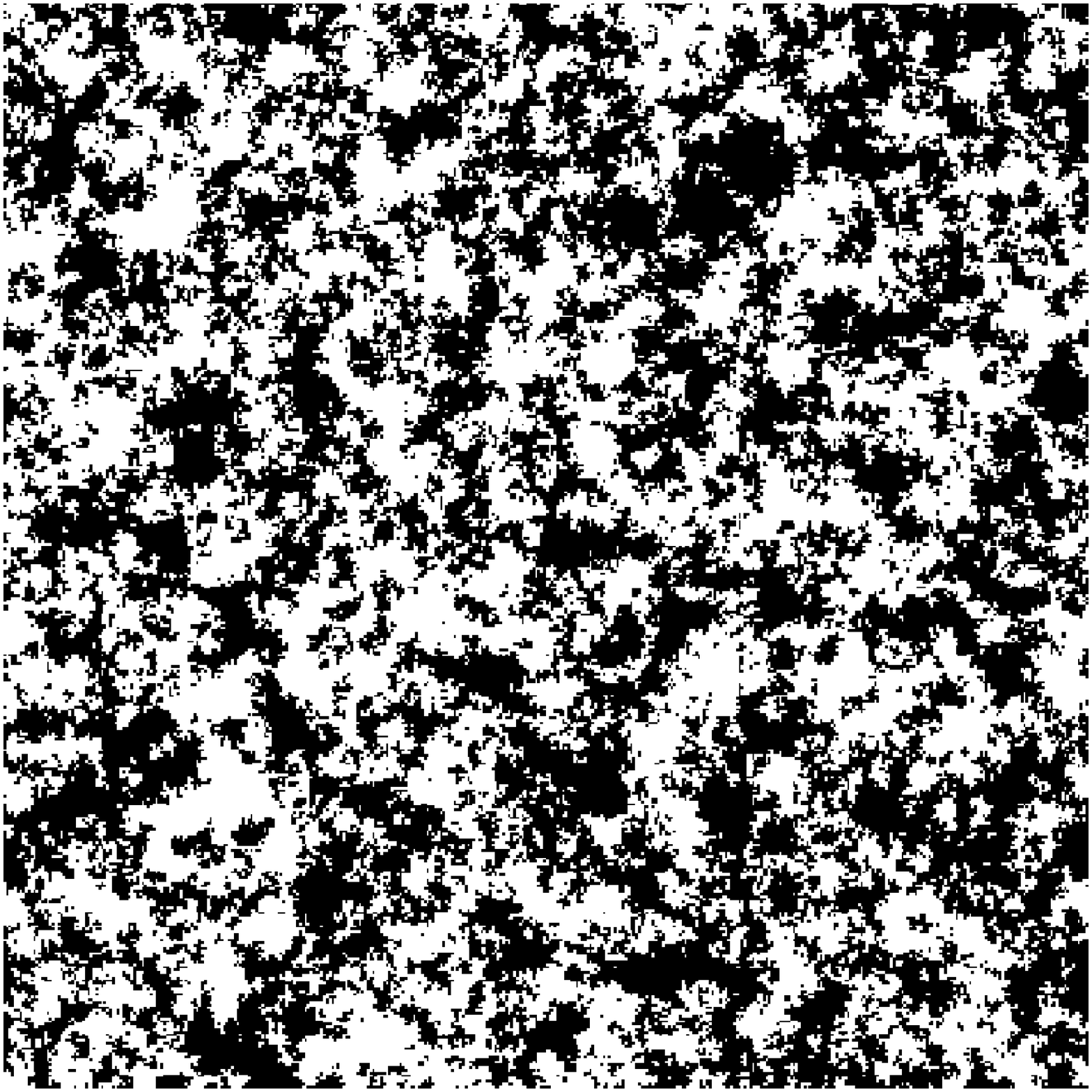, width = 220pt, height = 220pt}} \hspace{10pt}
      \subfigure{\epsfig{figure = 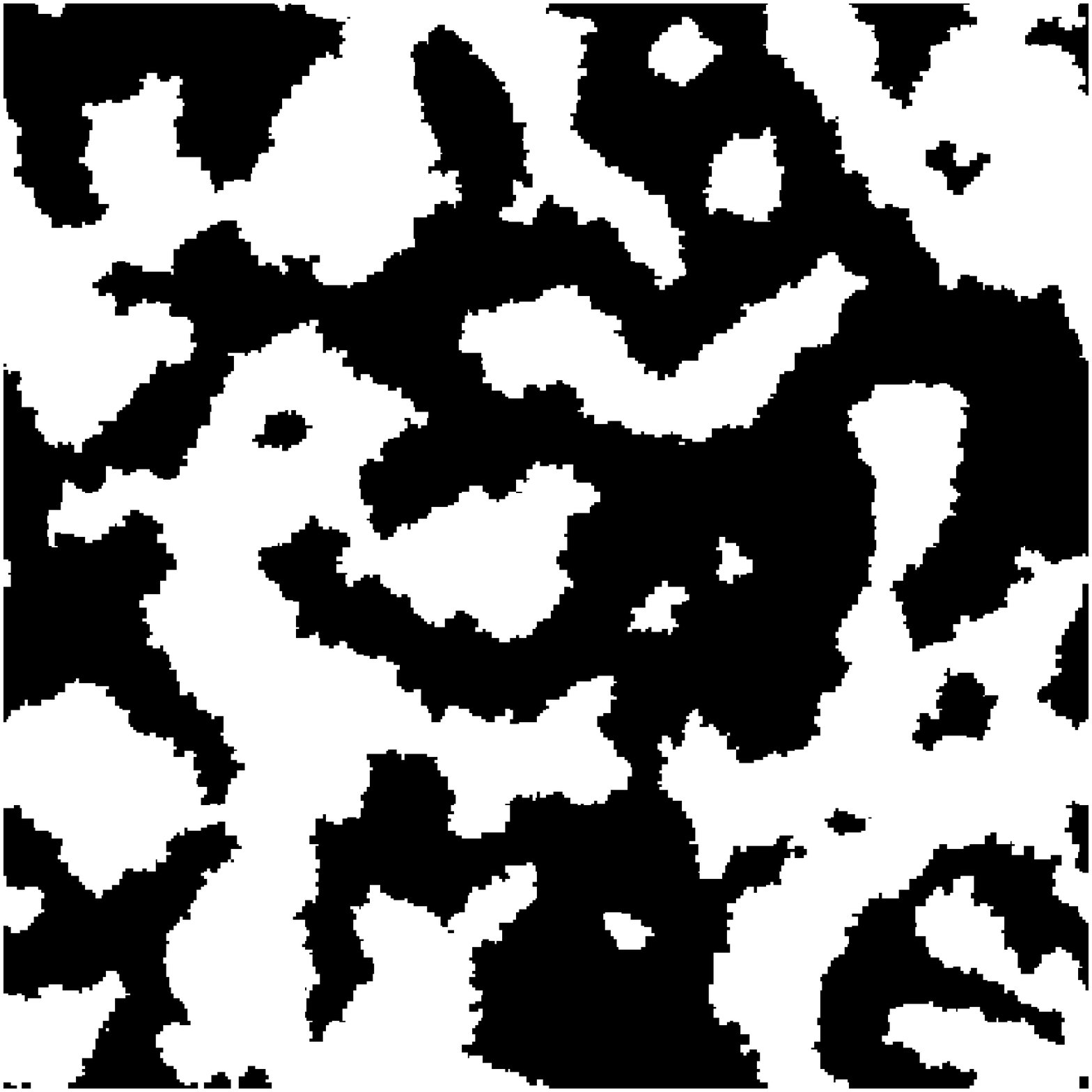, width = 220pt, height = 220pt}}}
\caption{\upshape{Simulation pictures of the voter model and majority rule model at time 20, respectively.
 Both processes evolve on a $400 \times 400$ lattice with periodic boundary conditions, and start from a Bernoulli product
 measure with an equal density of white and black vertices.}}
\label{fig:simulation}
\end{figure}

\section{The majority rule in space}
\label{sec:results}

\indent Numerical simulations of the majority rule model in one and two dimensions suggest that the asymptotic behavior of the
 process is somewhat similar to that of the voter model when $n$ is odd in the sense that the system clusters but similar to that
 of the biased voter model, i.e., the voter model modified so that individuals with opinion 0 update their opinion at a
 larger than individuals with opinion 1, when $n$ is even.
 As explained at the end of this section, our last theorem together with some heuristic arguments also suggests that, in contrast
 with the voter model, the majority rule model with $n$ odd clusters in any dimension.
 Hence, we state the following conjecture.
\begin{conjc} --
\label{co:odd-even}
 Clustering occurs when $n$ is odd, i.e., starting from any configuration,
 $$ P \,(\eta_t (x) \neq \eta_t (y)) \ \to \ 0 \ \ \hbox{as} \ \ t \to \infty \quad \hbox{for all} \ x, y \in \Z^d, $$
 whereas opinion 1 wins when $n$ is even, i.e., starting from any configuration that has infinitely many
 hyperedges with a majority of type 1 vertices,
 $$ P \,(\eta_t (x) = 0) \ \to \ 0 \ \ \hbox{as} \ \ t \to \infty \quad \hbox{for all} \ x \in \Z^d. $$
\end{conjc}
 The conclusion for $n$ even can be understood intuitively by observing that each tie in a discussion group, i.e., each tie at the time of the update
 of a hyperedge, results in a set of $n^d / 2$ type 0 vertices changing their opinion to type 1, while if the process is modified so that
 ties do not affect the configuration of the system, the rules are symmetric.
 The case when the parameter $n$ is odd is more interesting:
 in contrast with the results of Cox and Griffeath \cite{cox_griffeath_1986} which indicate that there is no natural scale for the
 asymptotics of the cluster size in the two-dimensional voter model, numerical simulations suggest that spatial correlations emerge much
 faster and that interface dynamics follow motion by mean curvature.
 This behavior is somewhat reminiscent of the behavior of some threshold voter models. % with intermediate thresholds.
 Threshold voter models denote a class of stochastic processes which, similarly to the voter model, describe opinion dynamics on
 the regular lattice.
 As in the voter model, individuals are characterized by one of two possible competing opinions and update their opinion independently
 at rate one based on their neighbors' opinion.
 But unlike in the voter model, the new opinion is not chosen uniformly at random from the neighborhood.
 Instead, individuals change their opinion if and only if the number of their neighbors with the opposite opinion exceeds a parameter
 $\theta$, called the threshold.
 The majority vote model is the special case in which the threshold $\theta$ equals half of the number of neighbors, therefore vertices
 are updated individually at rate one by adopting the majority opinion of their neighborhood.
 Although this model seems to be the spin system excluding the simultaneous update of several vertices the most closely related to our spatial
 version of the majority rule model, it does not cluster: instead, fixation occurs, i.e., any given individual stops changing opinion after a
 finite random time, as proved by Durrett and Steif \cite{durrett_steif_1993}.
 In fact, the behavior of the majority rule model is similar to that of threshold voter models with threshold parameter slightly
 smaller than half of the number of neighbors.

\indent The asymptotic behavior in the one-dimensional case is fully analyzed in this paper.
 Based on random walk estimates, we first prove that opinion 1 wins when $n$ is even.
\begin{theor} --
\label{th:even-1D}
 Assume that $d = 1$ and $n$ is even. Then, opinion 1 wins.
\end{theor}
 To study the process when $n$ is odd, we rely on duality techniques following the standard approach introduced for the voter model.
 In the case of the majority rule model, the set of vertices that one has to keep track to determine the opinion of a space-time
 point grows linearly going backwards in time, which is the main difficulty to establish clustering.
 The trick is to prove that the correlation between two space-time points only depends on a space-time region which is delimited by
 a semblance of the centers of the corresponding dual processes.
 Recurrence of symmetric random walks, which is the key to proving clustering of the voter model in one and two dimensions, is
 invoked in order to show that this space-time region is almost surely bounded.
\begin{theor} --
\label{th:odd-1D}
 Assume that $d = 1$ and $n$ is odd. Then, the process clusters.
\end{theor}

\indent The two-dimensional system is more difficult to study mainly because of the underlying hypergraph structure and the lack of
 mathematical tools in this context.
 Even though duality techniques are available in two dimensions as well, the dual process is hardly tractable due to an abundance
 of branching events.
 For simplicity, we mainly focus on the cases when $n = 2$ and $n = 3$ rather than the more general even/odd dichotomy.

\begin{figure}[t!]
\centering
\includegraphics[width=400pt]{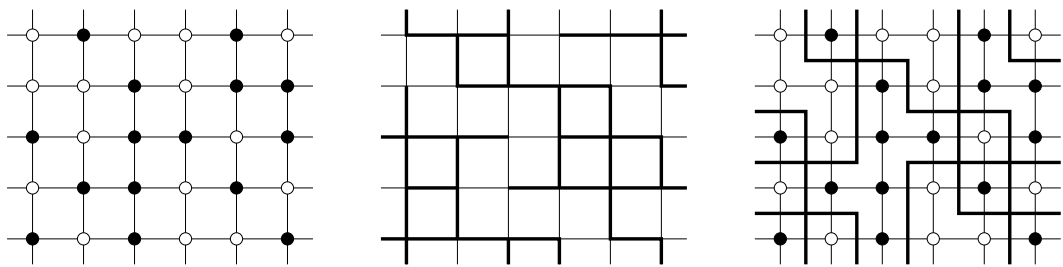}
\caption{\upshape{Dual representation between the majority rule model and the contour process}}
\label{fig:contour}
\end{figure}

\indent One key to analyzing the two-dimensional system is to look at a dual representation of the system that consists of keeping track
 of the disagreements along the edges of the lattice rather than the actual opinion at each vertex.
 This approach is partly motivated by the fact that the two-dimensional lattice seen as a planar graph is self dual.
 More precisely, we introduce a spin system coupled with the process and defined on the edge set by
 $$ \xi_t (e) \ = \ \xi_t ((x, y)) \ = \ \ind \,\{\eta_t (x) \neq \eta_t (y) \} \quad \hbox{for each edge} \ e = (x, y). $$
 To visualize the state space of this process, it is convenient to delete all the edges in state 0 and rotate all the edges in state 1
 of a quarter turn as shown in Figure \ref{fig:contour}.
 Motivated by the resulting picture, we call this process the \emph{contour} associated with the original spin system.
 This representation can as well be obtained by replacing every vertex $x \in \eta$ by the unit square centered at $x$ and taking the
 topological boundary of the union of these squares.
 We point out that the contour associated with a configuration $\eta$ can also be seen as a random subgraph of the dual lattice
 $$ \D^2 \ := \ \bigg\{x + \bigg(\frac{1}{2}, \frac{1}{2} \bigg) : x \in \Z^2 \bigg\}. $$
 Following the terminology of percolation theory, we say that an edge of the dual lattice is open if it belongs to the contour, and
 closed if it does not belong to the contour.

\indent To prove invasion of type 1 when $n = 2$, the key is to look at the majority rule model modified to have only type 0 outside a
 horizontal slice of height three and so that all vertices to the right of a type 0 also are in state 0.
 The reason for looking at such a process is that the profile of its contour can be simply characterized by a two-dimensional vector
 that keeps track of the distance between the rightmost type 1 vertices at all three levels.
 Moreover, the evolution rules of the distance between these vertices as well as the rate at which the vertices of
 type 1 are added to or removed from the system can be expressed in a simple manner based on certain geometric properties of the contour.
 Relying in addition on a block construction leads to the following result.

\begin{theor} --
\label{th:even-2D}
 Assume that $d = 2$ and $n = 2$. Then, opinion 1 wins.
\end{theor}

\indent Finally, we look at the two-dimensional majority rule when the set of hyperedges consists of the set of
 all three by three squares as a test model to understand the general case when $n$ is odd.
 To motivate our last result, consider the traditional voter model starting with a finite number of vertices of type~1.
 The process that keeps track of the number of type 1 vertices is a martingale since each time two vertices in different states interact,
 both vertices are equally likely to flip.
 In particular, the expected number of type 1 vertices is preserved by the dynamics, though the martingale convergence theorem implies
 almost sure extinction of the type 1 vertices.
%  The fact that the number of type 1 vertices stays constant in average is expected for neutral spin systems in general, where neutral
%  means that switching the state of all vertices does not modify the rate at which a given vertex flips.
 One of the most interesting aspects of the majority rule model, which again is reminiscent of threshold voter models with appropriate
 threshold parameter, is that, when starting from a finite initial configuration, the expected number of type 1 vertices is not constant.
 Our last theorem gives, for a class of configurations that we call regular clusters, an explicit expression of the variation rate of
 the number of type 1 as a function of the geometry of the cluster.
 This result supplemented with a heuristic argument strongly suggests that, as for some threshold voter models, the process with $n$
 odd clusters and explains the reason why the snapshot of the majority rule model on the right hand side of Figure \ref{fig:simulation}
 differs significantly from that of the voter model on the left hand side.
 To state our result, we need a few more definitions:
 given a configuration $\eta$, we call vertex $x$ a \emph{corner} whenever
 $$ \eta (x - e_1 - e_2) = \eta (x + e_1 + e_2) \neq \eta (x) \quad \hbox{or} \quad \eta (x - e_1 + e_2) = \eta (x + e_1 - e_2) \neq \eta (x) $$
 where $e_1$ and $e_2$ are the first and second unit vectors of the Euclidean plane.
 A corner $x$ is said to be a \emph{positive corner} if $\eta (x) = 1$ and a \emph{negative corner} if $\eta (x) = 0$.
 Also, we call $\eta$ a \emph{cluster} if its contour $\Gamma$ is a Jordan curve, i.e, a non-self-intersecting loop, and
 a \emph{regular cluster} if in addition
%  To define regular clusters, we set
%  $$ D_1 \ = \ \bigg[- \frac{1}{2}, \frac{1}{2} \bigg]^2 \quad \hbox{and} \quad D_2 \ = \ \bigg[- \frac{1}{2}, \frac{3}{2} \bigg]^2 $$
%  and for all $\eta : \Z^2 \longrightarrow \{0, 1 \}$, denote by $\Gamma = \Gamma (\eta)$ the topological boundary of
%  $$ \bigcup_{x \in \Z^2} \ \{x + D_1 : \eta (x) = 1 \}. $$
%  We call configuration $\eta$ a cluster if the boundary $\Gamma$ is a Jordan curve, and a regular cluster if in addition the
%  following two properties hold:
\begin{enumerate}
 \item the set $(x + D_2) \cap \Gamma$ is connected for all $x \in \D^2$ and \vspace{4pt}
 \item if vertex $x$ is a corner and $(x + D_2) \cap (y + D_2) \neq \varnothing$ then vertex $y$ is not a corner
\end{enumerate}
 where $D_2 = [- 1, 1]^2$.
 Without loss of generality, we assume that vertices located in the bounded region delimited by the Jordan curve are of type 1, which
 forces vertices in the unbounded region to be of type 0.
 Condition 1 above essentially says that the microscopic structure of the boundary of the cluster is not too complicated, i.e.,
 the Jordan curve does not zigzag too much, while condition~2 simply indicates that corners cannot be too close to each other.
 Finally, we let $c_+$ and $c_-$ denote the number of positive and negative corners, respectively, and let $\phi (\eta, x)$
 denote the variation of the number of type 1 vertices after the three by three square centered at $x$ is updated.
 Note that, when configuration $\eta$ is a cluster, $\phi (\eta, x) = 0$ for all but a finite number of vertices.

\begin{theor} --
\label{th:odd-2D}
 Assume that $\eta$ is a regular cluster with at least 11 vertices. Then
 $$ \sum_{x \in \Z^2} \ \phi (\eta, x) \ = \ 9 \,(c_- - c_+). $$
\end{theor}

\begin{figure}[t]
\centering
\includegraphics[width=400pt]{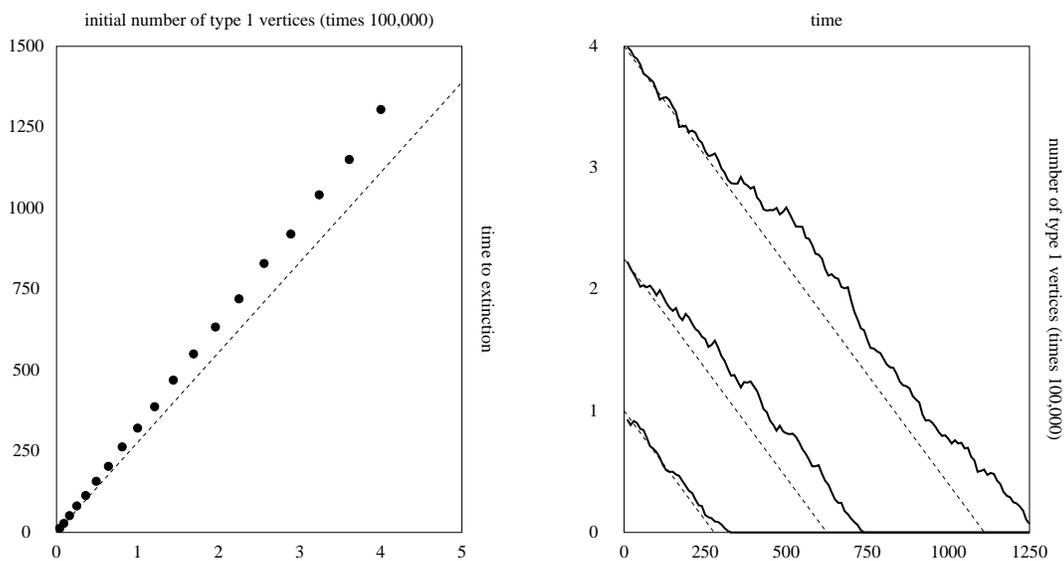}
\caption{\upshape{Simulation results for the process starting from a square cluster.
 The dots on the left picture give the time to extinction of type 1 vertices as a function of the initial number of type 1 averaged over
 100 realizations.
 The right picture gives the evolution of the number of type 1 vertices for three realizations.
 In both pictures, the dashed lines represent the corresponding expected values assuming a loss of 36 vertices of type 1 per unit of time.}}
\label{fig:extinction}
\end{figure}

\noindent In other words, the rate of variation of the number of type 1 vertices can be easily expressed as a function of the number
 of positive and negative corners, which directly implies that the expected number of type 1 vertices is not constant.

\indent To conclude this section, we give some heuristic arguments which, together with the previous theorem, supports the first part
 of Conjecture \ref{co:odd-even} when $n$ is odd.
 The main purpose of Theorem \ref{th:odd-2D} is to support the idea that the time to extinction of a finite cluster scales like the
 original size of the cluster.
 To this extent, the assumption that the cluster must have at least 11 vertices is not a limitation since clusters with an even smaller
 size are destroyed quickly.
 To relate the theorem to the time to extinction of the type 1 opinion, we first observe that, traveling around the Jordan curve clockwise,
 the number of right turns is equal to the number of left turns plus four.
 The result directly follows by using a simple induction over the number of type 1 vertices.
 This suggests that, when averaged over time from time 0 to the time to extinction, the difference between the number of
 negative and positive corners should be about $-4$, further suggesting that the time to extinction is equal to about the initial
 number of type 1 vertices divided by $9 \times 4 = 36$.
 Figure \ref{fig:extinction} compares our speculative argument with simulation results for the process starting with a square cluster.
 Even though these do not fit perfectly, the numerical results strongly suggest that the time to extinction is indeed linear
 in the initial number of type 1 vertices, % when starting from a square cluster,
 which drastically contrasts with the voter model, and that our 36 is not far from the truth.
 This heuristic argument also indicates that the majority rule dynamics quickly destroy small clusters, thus resulting in a
 clustering more pronounced than in the voter model, which explains the striking difference between the two pictures
 of Figure \ref{fig:simulation}.
 Finally, we point out that the intuitive ideas behind the proof of Theorem \ref{th:odd-2D} are not sensitive to the
 spatial dimension.
 Also, we conjecture that the theorem holds in higher dimensions with the constant 9 replaced by $3^d$, and that the
 majority rule model with $n$ odd clusters in any spatial dimensions, as mentioned at the beginning of this section
 in Conjecture \ref{co:odd-even}.

%%%%%%%%%%%%%%%%%%%%%%%%%%%%%%%%%%%%%%%%%%%%%%%%%%%%%%%%%%%%%%%%%%%%%%%%%%%%%%%%%%%%%%%%%%%%%%%%%%%%%%%%%%%%%%%%%%%%%%%%%%%%%%%%%%%%%%%%%%

\section{Proof of Theorems \ref{th:even-1D} and \ref{th:odd-1D} ($d = 1$)}
\label{sec:even-1D}

\indent This section is devoted to the analysis of the one-dimensional majority rule model for which we prove that clustering occurs
 for all sizes $n$ of the hyperedges while opinion 1 wins under the additional assumption that $n$ is even.
 The latter is based on simple random walk estimates while the former further relies on duality techniques, which consists of keeping
 track of the ancestry of a finite number of vertices going backwards in time.

\indent In order to define the dual process, the first step is to construct the process graphically.
 The construction is similar in any spatial dimension.
 To each hyperedge $h_x$ we attach a Poisson process with parameter 1 whose $j$th arrival time is denoted by $T_j (x)$.
 Poisson processes attached to different hyperedges are independent.
 At time $t = T_j (x)$, we have the following alternative:
\begin{enumerate}
 \item If $\card \,(\eta_{t-} \cap h_x) < n^d / 2$ then all vertices in $h_x$ become of type 0. \vspace{4pt}
 \item If $\card \,(\eta_{t-} \cap h_x) \geq n^d / 2$ then all vertices in $h_x$ become of type 1.
\end{enumerate}
 Results due to Harris \cite{harris_1972} which apply to traditional interacting particle systems on lattices but extend directly to
 the hypergraph $\mathbb H$ guarantee that the majority rule model starting from any initial configuration $\eta_0 \subset \Z^d$ can
 be constructed using the collection of independent Poisson processes and the majority rule at the arrival times introduced above.
 To visualize this in one dimension, we draw a line segment from vertex $x$ to vertex $x + n - 1$ at the arrival times of the Poisson
 process attached to the hyperedge $h_x$.
 Note that this line segment connects all the vertices in $h_x$.
 To study the process when $n$ is even, we first assume that $\eta_0 = (- \infty, 0] \cap \Z$. Therefore
 $$ \eta_t \ = \ (- \infty, X_t] \,\cap \,\Z \ \ \hbox{for all} \ t > 0 \quad \hbox{where} \quad  X_t \ := \ \max \,\{x \in \Z : x \in \eta_t \}. $$
 The key to proving Theorem \ref{th:even-1D} is the following lemma.
\begin{lemma} --
\label{lem:drift}
 With probability one, $X_t \to \infty$ as $t \to \infty$.
\end{lemma}
\begin{proof}
 Note that there are exactly $n$ hyperedges that contain vertex $X_t$ therefore $n - 1$ possible events that affect the position
 of the rightmost 1.
 From the leftmost to the rightmost, these updates create/remove respectively the following numbers of type 1 vertices:
 $$ \hbox{create} \ \ 1, \ 2, \ \ldots \ , \ \frac{n}{2} - 1, \ \frac{n}{2} \qquad \hbox{remove} \ \ \frac{n}{2} - 1, \ \frac{n}{2} - 2, \ \ldots, \ 2, \ 1. $$
 In other words, we have the transition rates
 $$ X_t \ \to \ X_t + j \quad \hbox{at rate one} \quad \hbox{for all} \ \ j \in \bigg\{1 - \frac{n}{2}, \ 2 - \frac{n}{2}, \ \ldots \ , \ \frac{n}{2} \bigg\}. $$
 Summing over all the possible values of the increment, we get
 $$ E \,(X_1 - X_0) \ = \ \sum_{j = 1}^n \ \bigg(j - \frac{n}{2} \bigg) \ = \ \frac{n (n + 1)}{2} - \frac{n^2}{2} \ = \ \frac{n}{2} \ > \ 0. $$
 The expected value can be understood intuitively as follows.
 There are $n - 2$ updates that can be paired off in such a way that each pair consists of one update that causes $k$ vertices of type 0
 to flip and one update that causes $k$ vertices of type 1 to flip.
 The remaining update corresponds to a tie that causes $n/2$ vertices of type 0 to flip, which gives the expected value above.
 In particular, an application of the Law of Large Numbers implies that $X_t$ converges almost surely to infinity as time goes to infinity,
 which completes the proof.
\end{proof} \\ \\
 It is straightforward to deduce from Lemma \ref{lem:drift} that
 $$ \begin{array}{l}
  P \,(\eta_t (x) \to 1 \ \hbox{as} \ t \to \infty \ \hbox{for all} \ x \in \Z \ | \,\eta_0 = h_0) \vspace{4pt} \\ \hspace{40pt} \geq \
  P \,(X_t - X_0 \geq 0 \ \hbox{for all} \ t \geq 0) \times P \,(X_0 - X_t \leq 0 \ \hbox{for all} \ t \geq 0) \ > \ 0. \end{array} $$
 Theorem \ref{th:even-1D} follows directly from the previous estimate since the latter implies that, starting with infinitely many
 hyperedges  with a majority of type 1, there exists with probability one a cluster of vertices of type 1 that expands indefinitely.

\indent We now turn to the proof of Theorem \ref{th:odd-1D} which relies on duality techniques.
 The ancestry of a given space-time point, i.e., the set of vertices at earlier times that determine the opinion of the point under consideration,
 grows linearly going backwards in time.
 While the whole structure of the ancestry, which keeps growing indefinitely, is necessary to determine the opinion of a given vertex based on the
 initial configuration, given two vertices, only a finite space-time region is relevant in proving that they share ultimately the same opinion.
 In order to define this space-time region and the dual process starting at a given point, we first introduce
 $$ T (u) \ = \ \{T_j (u) : j \geq 1 \} \quad \hbox{and} \quad c (u) \ = \ u + \frac{n - 1}{2} \quad \hbox{for all} \ u \in \Z. $$
 Note that $c (u)$ is simply the center of the hyperedge $h_u$.
 The dual process starting at a given space-time point $(x, T)$ is the set-valued process initiated at $\hat \eta_0 (x, T) = \{x \}$ and defined
 recursively as follows:
 assuming that the dual process has been defined up to time $s$, we let
 $$ \tau (s) \ = \ T \ - \ \sup \,\{T (v) \cap (0, T - s) : v \in \Z \ \hbox{and} \ \hat \eta_s (x, T) \cap h_v \neq \varnothing \}. $$
 There is a unique vertex $w \in \Z$ such that $T - \tau (s) \in T (w)$.
 Then, we define
 $$ \hat \eta_t (x, T) \ = \ \hat \eta_s (x, T) \ \ \hbox{for all} \ \ t \in (s, \tau (s)) \quad \hbox{and} \quad
    \hat \eta_{\tau (s)} (x, T) \ = \ \hat \eta_s (x, T) \,\cup \,h_w. $$
 In words, going backwards in time, each time the dual process ``encounters'' a line segment in the graphical representation, the corresponding
 hyperedge is added to the process.
 Therefore, the dual process consists of an interval of vertices that grows linearly going backwards in time.
 The graphical representation restricted to the space-time region induced by the dual process together with the initial configuration
 in $\hat \eta_T (x, T)$ allows to determine the opinion of $(x, T)$.
 However, we can prove that two given vertices share ultimately the same opinion without looking at their opinion or the whole
 structure of their dual processes.
 To do so, we define a new process $c_s (x, T)$ that we shall call the \emph{center path} of space-time point $(x, T)$.
 Again, $c_0 (x, T) = x$ and the process is defined recursively based on the Poisson events:
 assuming that the path has been defined until time $s$, let
 $$ \sigma (s) \ = \ T \ - \ \sup \,\{T (v) \cap (0, T - s) : v \in \Z \ \hbox{and} \ c_s (x, T) \in h_v \}. $$
 There is a unique vertex $w \in \Z$ such that $T - \sigma (s) \in T (w)$ and we define
 $$ c_t (x, T) \ = \ c_s (x, T) \ \ \hbox{for all} \ \ t \in (s, \sigma (s)) \quad \hbox{and} \quad
    c_{\sigma (s)} (x, T) \ = \ c (w). $$
 In words, going backwards in time, each time the center path ``encounters'' a line segment in the graphical representation, it jumps to the
 center of this line segment.
 To complete the construction, we now let $x < y$ be two vertices, and define the space-time region $\Omega$ which is delimited by their
 respective center paths by setting
 $$ \begin{array}{rcl}
         S & = & \inf \,\{s > 0 : c_s (x, T) = c_s (y, T) \} \vspace{4pt} \\
    \Omega & = & \{(z, t) \in \Z \times (\max (T - S, 0), T) : c_{T - t} (x, T) \leq z \leq c_{T - t} (y, T) \}. \end{array} $$
 We refer to Figure \ref{fig:dual} for a picture where $\Omega$ is represented by the hatched polygonal region.
 The key to proving Theorem \ref{th:odd-1D} is that, provided the center paths intersect by time 0, all space-time points in the region
 $\Omega$ share the same opinion, which is established in the following lemma.

\begin{figure}[t!]
\centering
\includegraphics[width=280pt]{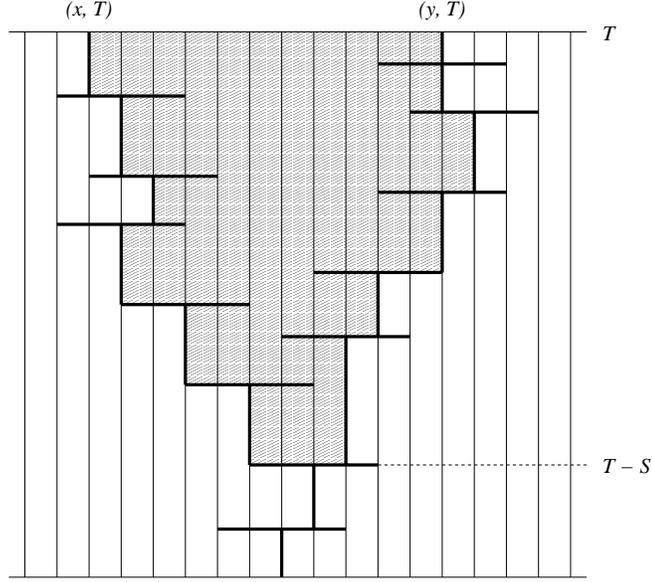}
\caption{\upshape{Picture of the center paths when $n = 5$}}
\label{fig:dual}
\end{figure}

\begin{lemma} --
\label{lem:common}
 Assume that $S < T$.
 Then, the function $\Phi (z, t) := \eta_t (z)$ is constant on $\Omega$.
\end{lemma}
\begin{proof}
 Define $\Lambda = \{(u, t) \in \Z \times \R_+ : t \in T (u) \}$ and the collection
 $$ H^{\star} \ = \ \{h (u, t) := h_u \times \{t \} : (u, t) \in \Lambda \ \hbox{and} \ h (u, t) \cap \Omega \neq \varnothing \} $$
 which can be seen as the set of all line segments of the graphical representation that intersect the space-time region $\Omega$.
 Note that this corresponds to the set of all Poisson events that may affect the configuration of the process in $\Omega$.
 First, by definition of $S$, there is $h (u, t) \in H^{\star}$ such that
 $$ t \ = \ T - S \quad \hbox{and} \quad c_{S-} (x, T), c_{S-} (y, T) \in h_u $$
 from which it follows that $\Phi$ is constant on $\Omega \cap (\Z \times \{T - S \})$ and equal to the majority type in the hyperedge
 $h_u$ at time $T - S$.
 To prove that this property is retained at later times, let
 $$ h (u, t) \ \in \ H^{\star} \quad \hbox{such that} \quad s := T - t \neq S $$
 and observe that we have the following alternative:
\begin{enumerate}
 \item $c_{s-} (x, T), c_{s-} (y, T) \notin h_u$ and then
  $$ h_u \ \subset \ (c_s (x, T), c_s (y, T)) \,\cap \,\Z \quad \hbox{and} \quad \card \,\{h_u \cap [c_s (x, T), c_s (y, T)] \} \ = \ n. $$
 \item $c_{s-} (x, T) \in h_u$ and then $c_s (x, T) = c (u)$ and
  $$ \card \,\{h_u \cap [c_s (x, T), c_s (y, T)] \} \ = \ \card \,\{c (u), c (u) + 1, \ldots, u + n - 1 \} \ > \ n/2. $$
 \item $c_{s-} (y, T) \in h_u$ and then $c_s (y, T) = c (u)$ and
  $$ \card \,\{h_u \cap [c_s (x, T), c_s (y, T)] \} \ = \ \card \,\{u, u + 1, \ldots, c (u) \} \ > \ n/2. $$
\end{enumerate}
 In all three cases, we have that
 $$ \eta_{t-} (z) = i \ \ \hbox{for} \ c_s (x, T) \leq z \leq c_s (y, T) \quad \hbox{implies} \quad \card \,\{z \in h_u : \eta_{t-} (z) = i \} \ > \ n/2 $$
 from which it follows that $\eta_t (z) = i$ for all $c_{s-} (x, T) \leq z \leq c_{s-} (y, T)$.
 This indicates that the property to be proved is retained going forward in time through the Poisson events in $H^{\star}$.
 Since the other Poisson events do not affect the space-time region $\Omega$, the lemma follows.
\end{proof} \\ \\
 In view of the previous lemma, we have $\eta_T (x) = \eta_T (y)$ whenever $S < T$.
 In other respect, the same argument as in Lemma \ref{lem:drift} implies that both center paths evolve according to independent symmetric
 random walks until they intersect.
 More precisely,
 $$ c_s (x, T) \ \to \ c_s (x + T) + j \quad \hbox{at rate one} \quad \hbox{for all} \ \ j \in \bigg\{ - \frac{n - 1}{2}, \ \ldots \ , \ \frac{n - 1}{2} \bigg\}. $$
 Since symmetric random walks are recurrent in one dimension, the probability that they intersect by time 0, that is $S < T$,
 approaches one as time $T \to \infty$.
 This proves Theorem \ref{th:odd-1D}.

%%%%%%%%%%%%%%%%%%%%%%%%%%%%%%%%%%%%%%%%%%%%%%%%%%%%%%%%%%%%%%%%%%%%%%%%%%%%%%%%%%%%%%%%%%%%%%%%%%%%%%%%%%%%%%%%%%%%%%%%%%%%%%%%%%%%%%%%%%

\section{Proof of Theorem \ref{th:even-2D} ($d = 2$ and $n = 2$)}
\label{sec:even-2D}

\noindent This section is devoted to Theorem \ref{th:even-2D} whose proof is based on a rescaling argument.
 This technique is also known as block construction and was introduced by Bramson and Durrett \cite{bramson_durrett_1988} and further
 refined by Durrett \cite{durrett_1995}.
 Even though the block construction is now a standard tool in the field of interacting particle systems, its application is rarely
 straightforward and requires additional nonstandard arguments, especially in the case of the majority rule model.

\indent In preparation for the application of a block construction, we first investigate a new process that we shall
 call the \emph{slice process} which is the 2-dimensional 4-neighborhood majority rule model modified in the following two ways.
 First, the process is restricted to the horizontal slice
 $$ S_3 \ = \ \{x = (x_1, x_2) \in \Z^2 : |x_2| \leq 1 \} $$
 in the sense that all vertices in the complement of $S_3$ are unchangeably in state 0.
 Second, the process is modified so that all vertices to the right of a vertex in state 0 and with the same second
 coordinate flip instantaneously to state 0, which implies that updates that result in the existence of a vertex in state 1
 to the right of a vertex in state 0 are suppressed.
 In particular, state 1 is instantaneously driven to extinction when starting from a random initial condition for which
 $$ P \,(\hbox{for all $(x_1, x_2) \in S_3$ there exists $z_1 \leq x_1$ such that $(z_1, x_2)$ is in state 0}) \ = \ 1. $$
 Therefore, to avoid trivialities, we assume that the slice process starts from the deterministic configuration in which
 all the vertices in the horizontal slice $S_3$ with a nonpositive first coordinate are in state 1 and all other vertices
 are in state 0.
 Although our verbal description of the slice process is probably clear enough, for the sake of rigor we also give its
 Markov generator
%  $$ \begin{array}{rcl} L_3 f (\bar \eta) & = &
%   \displaystyle \sum_{x_1 \in \Z} \ \sum_{x_2 \in \Z}  \ [f (\bar \eta_{S_x, 0}) - f (\bar \eta)] \ \times \ \ind \ \bigg\{\sum_{y \in S_x} \bar \eta (y) < 2 \bigg\} \vspace{8pt} \\ & + &
%   \displaystyle \sum_{x_1 \in \Z} \ \sum_{x_2 = -1, 0} \ [f (\bar \eta_{S_x, 1}) - f (\bar \eta)] \ \times \ \ind \ \bigg\{\sum_{y \in S_x} \bar \eta (y) \geq 2 \bigg\} \
%   \prod_{z_1 < x_1} \prod_{z_2 = 0, 1} \bar \eta ((z_1, x_2 + z_2)) \end{array} $$
%  $$ \begin{array}{l} L_3 f (\bar \eta) \ = \
%   \displaystyle \sum_{x_1 \in \Z} \ \sum_{x_2 \in \Z}  \ \ind \,\{\card (\bar \eta \cap h_x) < 2 \} \ [f (\bar \eta \setminus h_x) - f (\bar \eta)] \vspace{8pt} \\ \hspace{20pt} + \
%   \displaystyle \sum_{x_1 \in \Z} \ \sum_{x_2 = -1, 0} \ \prod_{z_1 < x_1} \prod_{z_2 = 0, 1} \bar \eta ((z_1, x_2 + z_2)) \
%   \displaystyle \ind \,\{\card (\bar \eta \cap h_x) \geq 2 \} \ [f (\bar \eta \cup h_x) - f (\bar \eta)] \ \end{array} $$
 $$ \begin{array}{l} L_3 f (\bar \eta) \ = \
  \displaystyle \sum_{x} \ \ind \,\{\card \,(\bar \eta \cap h_x) < 2 \} \ [f (\bar \eta \setminus h_x) - f (\bar \eta)] \vspace{2pt} \\ \hspace{10pt} + \
  \displaystyle \sum_{x} \ \ind \,\{\card \,(\bar \eta \cap h_x) \geq 2, \ h_x \subset S_3, \ (- \infty, x_1) \times \{x_2, x_2 + 1 \} \subset \bar \eta \} \
  \displaystyle [f (\bar \eta \cup h_x) - f (\bar \eta)] \ \end{array} $$
 where $x_1$ and $x_2$ denote the first and second coordinates of vertex $x \in \Z^2$.
 Note that the slice process is stochastically smaller than the original majority rule model, i.e., the processes starting from the
 same initial configuration can be coupled in such a way that, at all times, the set of type 1 vertices of the majority rule model contains
 the set of type 1 vertices of the slice process.
 The reason for introducing the two modification rules that define the slice process is that they simplify the dynamics to make
 them more tractable mathematically without however preventing opinion 1 from invading the slice $S_3$ so that Theorem \ref{th:even-2D}
 can be eventually deduced from stochastic domination and a block construction.
%  In order to prove Theorem 13, we will introduce a new process that is stochastically less than the 2-dimensional 4-neighborhood
%  majority rule model.  We will define the slice process as the majority rule model on a slice of $\Z^2$ with infinite length and width 5.
%  The top and bottom strips are unchangeably state 0, and state 1 can invade the middle of the slice.
%  Also, to simplify the process further, a site cannot change to state 1 if any sites to the right of it are of state 0.
% \begin{defin}
%  The Slice process
%  $$ \begin{array}{rcl} L f (\eta) & = &
%   \displaystyle \sum_{x_1 \in \Z} \sum_{x_2 \in [4]} \ [f (\eta_{S_{x_1,x_2}, 0}) - f (\eta)] \ \times \
%   \ind \ \bigg\{\sum_{y \in S_{x_1,x_2}} \eta (y) < 2 \bigg\}  + \
%   \displaystyle \sum_{x_1 \in \Z}  \sum_{x_2 \in [4]} \ [f (\eta_{S_{x_1,x_2}, 1}) - f (\eta)] \ \times \\ && \
%   \ind \ \bigg\{\sum_{y \in S_{x_1,x_2}} \eta (y) \geq 2 \hspace{2pt} \bigcap \hspace{2pt}
%    x_2 \in \{2,3\} \hspace{2pt} \bigcap \hspace{2pt} \forall i<x_1, \eta(i,x_2) + \eta(i,x_2+1)=0 \bigg\}
% \end{array} $$
% \end{defin}

%  The process can also be defined through the rates $X^+_t$, $X_t$, and $X^-_t$ which are defined by the x-coordinate of the farthest site from the y axis of state 1 with y-coordinate 1, 0, and -1 respectively. %Fix- these aren't analogous to the previous definition's coordinates.
%  $$ \begin{array}{lcl}
%  X^+_t = \ind(X_t\ge X^+_t) - \ind(X_t<X^+_t) - 1 \\
%  X_t = \ind(X^+_t\ge X_t) + \ind(X^-_t\ge X_t) - \ind(X^-_t<X_t) - \ind(X^+_t<X_t)\\
%  X^-_t = \ind(X_t\ge X^-_t) - \ind(X_t<X^-_t) - 1 \end{array} $$

\indent To investigate the slice process and prove that it invades the slice $S_3$ we note that the second modification rule implies
 that, for $x_2 = -1, 0, 1$, all vertices with second coordinate $x_2$ to the left of the rightmost vertex in state 1 and also with
 second coordinate $x_2$ are in state 1.
 In particular, the configuration of the slice process is uniquely defined by the position of its three rightmost vertices in
 state 1 with second coordinate $-1, 0, 1$, or equivalently the Markov process
 $$ X_t = (X_t (x_2) : x_2 = -1, 0, 1) \ \ \hbox{where} \ \ X_t (x_2) = \max \,\{x_1 : \bar \eta_t ((x_1, x_2)) = 1 \}. $$
 Note also that the dynamics of the slice process induced by the 4-neighborhood majority rule model imply that the middle component of
 the process $X_t$ cannot be simultaneously smaller than the other two components.
 By invasion in the slice we mean almost sure convergence of all three components to infinity.
 We first introduce the following functional associated to the slice process:
 $$ \begin{array}{rcl}
      D (\Sigma_t) & = & \lim_{h \to 0} \ h^{-1} \,(\Sigma_{t + h} - \Sigma_t) \ \ \hbox{where} \ \ \Sigma_t = X_t (-1) + X_t (0) + X_t (1), \vspace{4pt} \\
           D (G_t) & = & \lim_{h \to 0} \ h^{-1} \,(G_{t + h} - G_t) \ \ \hbox{where} \ \ G_t = |X_t (1) - X_t (0)| + |X_t (-1) - X_t (0)| \end{array} $$
 that we call the \emph{sum's drift} and the \emph{gap's drift}, respectively.
 The analysis of these two processes indicate that the sum $\Sigma_t$ drifts to infinity whereas the gap $G_t$ is uniformly
 bounded in time, from which it follows that all three components of $X_t$ converge almost surely to infinity.
 The analysis of the sum's and gap's drifts relies on asymptotic properties of the functional
 $$ \iota (X_t) \ = \ (X_t^+, X_t^-) \ \ \hbox{where} \ \ X_t^+ = X_t (1) - X_t (0) \ \ \hbox{and} \ \ X_t^- = X_t (-1) - X_t (0) $$
 that we shall call for obvious reasons the \emph{interface process}.
 Letting $(a, b)$ denotes the state of the interface, we always have $a \leq 0$ or $b \leq 0$ because the middle component of the process
 $X_t$ cannot be simultaneously smaller than its other two components.
 We also observe that the value of the sum's drift and the value of the gap's drift are not affected by the symmetry about the $x$-axis.
 In particular, we identify interfaces that can be deduced from one another by this axial symmetry, that is we identify states $(a, b)$
 and $(b, a)$.
 Therefore, the interface process can be seen as a continuous-time random walk on a certain connected graph with vertex set
 $$ V \ = \ \{(a, b) : a \leq 0 \ \hbox{and} \ a \leq b \}. $$
 We have represented the transition rates of the interface process on a portion of this connected graph around vertex $(0, 0) \in V$
 in Figure \ref{fig:drift}.
 Information about the dynamics of the interface given in this figure are employed frequently in some of the following lemmas.
%  The first step is to prove important properties of the drift and interface processes at equilibrium, which will be used to deduce 
%  almost sure convergence of $\Sigma_t$ to infinity.
%  Then, we will show that the rightmost type 1 particles at level $-1, 0, 1$ are attracted to each other so that properties of the
%  process $\Sigma_t$ extend to all three components of the process $X_t$.
%  We now make these arguments precise.

\begin{figure}[t!]
\centering
\includegraphics[width=400pt]{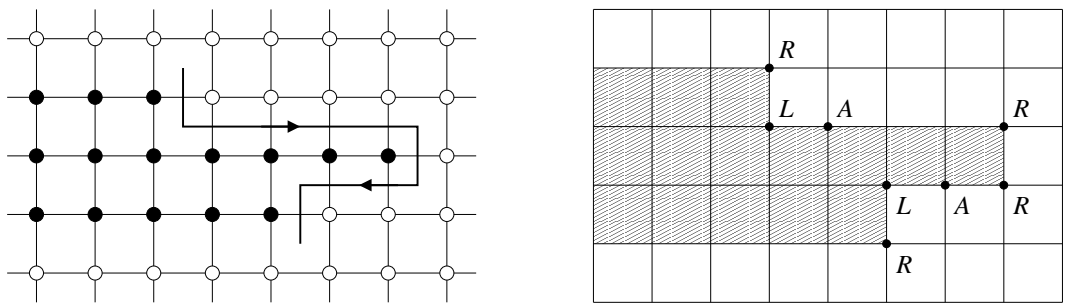}
\caption{\upshape{Pictures related to the proof of Lemma \ref{lem:drift-1}}}
\label{fig:drift-1}
\end{figure}

\begin{lemma} --
\label{lem:drift-1}
 We have $D (\Sigma_t) = 2 \,(N (X_t) - 1)$ where $N (X_t) = \ind \,\{|X_t^+| \neq 1 \} + \ind \,\{|X_t^-| \neq 1 \}$.
\end{lemma}
\begin{proof}
 The proof relies on a series of simple geometric arguments.
 First of all, we observe that each point $z$ of the dual lattice of $\Z^2$ has exactly four nearest neighbors in $\Z^2$.
 These four neighbors define a size 4 neighborhood that is updated at rate 1 and that we call the neighborhood with center $z$.
 For each configuration $\bar \eta$ of the slice process, we let $\Gamma = \Gamma (\bar \eta)$ be the contour associated with
 $\bar \eta$ defined as in Section \ref{sec:results}.
 Note that this contour is a doubly infinite self-avoiding path on the dual lattice.
 We let $\vec{\gamma}$ be the finite portion of this path that connects the points
 $$ \gamma^+ \ = \ \bigg(X_t ( 1) + \frac{1}{2},   \frac{3}{2} \bigg) \quad \hbox{and} \quad
    \gamma^- \ = \ \bigg(X_t (-1) + \frac{1}{2}, - \frac{3}{2} \bigg) $$
 and orient this portion from point $\gamma^+$ to point $\gamma^-$.
 The first picture of Figure \ref{fig:drift-1} gives an example of configuration of the slice process where black dots refer to
 vertices in state 1 and white dots to vertices in state 0, together with the corresponding oriented path $\vec{\gamma}$
 represented in thick lines.
 Note that the oriented path $\Gamma (\bar \eta_t)$ has exactly $G_t + 4$ vertices so we write
 $$ \vec{\gamma} \ = \ (\gamma (1) = \gamma^+, \gamma (2), \gamma (3), \ldots, \gamma (G_t + 4) = \gamma^-) $$
 in the direction of the orientation, and let $\ep (j)$ be the edge connecting $\gamma (j)$ and $\gamma (j + 1)$.
 Note also that any update in a neighborhood whose center does not belong to the oriented path does not yield any change in the configuration
 of the slice process, either because this neighborhood already contains four vertices in the same state, or because it contains two vertices
 in each state but is not included in $S_3$.
 To compute the drift of $\Sigma_t$, we introduce the following classification.
\begin{enumerate}
 \item Point $z \in \vec{\gamma}$ is called a \emph{right turn} if the neighborhood with center $z$ contains exactly one vertex in state 1.
  In this case, an update of the slice process in this neighborhood always results in one vertex changing from state 1 to state 0. \vspace{2pt}
 \item Point $z \in \vec{\gamma}$ is called a \emph{straight point} if the neighborhood with center $z$ contains exactly two vertices
  in state 1 and two vertices in state 0.
 \begin{enumerate}
  \item The straight point is said to be \emph{active} if the neighborhood with center $z - e_1$ contains three or four vertices in state 1,
   in which case an update in the neighborhood with center $z$ results in two vertices changing from state 0 to state 1.
  \item Otherwise, the straight point is said to be \emph{inactive}, in which case an update in the neighborhood with center $z$ does not yield
   any change in the configuration of the slice process due to the second modification rule.
 \end{enumerate}
 \item Point $z \in \vec{\gamma}$ is called a \emph{left turn} if the neighborhood with center $z$ contains exactly three vertices in state 1.
  In this case, an update of the slice process in this neighborhood always results in one vertex changing from state 0 to state 1.
\end{enumerate}
 In the second picture of Figure \ref{fig:drift-1}, right turns, active straight points, and left turns corresponding to the configuration in
 the first picture are marked with the letters $R, A, L$, respectively.
 Since neighborhoods are updated independently and at rate one, the drift can be computed based on the number of left/right turns and active
 straight points.
 To count the number of points in each class, we first observe that $\gamma^+$ and $\gamma^-$ are always right turns, while to determine the
 class of the other points, we distinguish between the following two cases.
\begin{enumerate}
 \item In case $X_t^+ = 0$, point $\gamma (2) = \gamma (|X_t^+| + 2)$ is an active straight point. \vspace{2pt}
 \item In case $X_t^+ \neq 0$, we observe that edges $e (1)$ and $e (|X_t^+| + 2)$ are downwards vertical edges while intermediate edges are
  horizontal edges.
  This implies that $\gamma (2)$ and $\gamma (|X_t^+| + 2)$ are turns with opposite directions, and that intermediate points are straight points.
\end{enumerate}
 The class of points $\gamma (j)$, $j = |X_t^+| + 3, \ldots, G_t + 3$, can be determined similarly.
 In particular, the number of right turns minus the number of left turns always equals two, which allows to quantify the drift
 in terms of the number of active straight points exclusively:
 $$ D (\Sigma_t) \ = \ 2 \times \hbox{number of active straight points} - 2. $$
 Finally, there is one active straight point with second coordinate $1/2$ if and only if $X_t^+ \neq 1$, and one with second coordinate $-1/2$
 if and only if $X_t^- \neq 1$.
%  similarly for points at level $-1/2$.
 Therefore, we conclude that the number of active straight points is simply equal to $N (X_t)$, which completes the proof of the Lemma.
\end{proof}

\begin{figure}[t!]
\centering
\includegraphics[width=400pt]{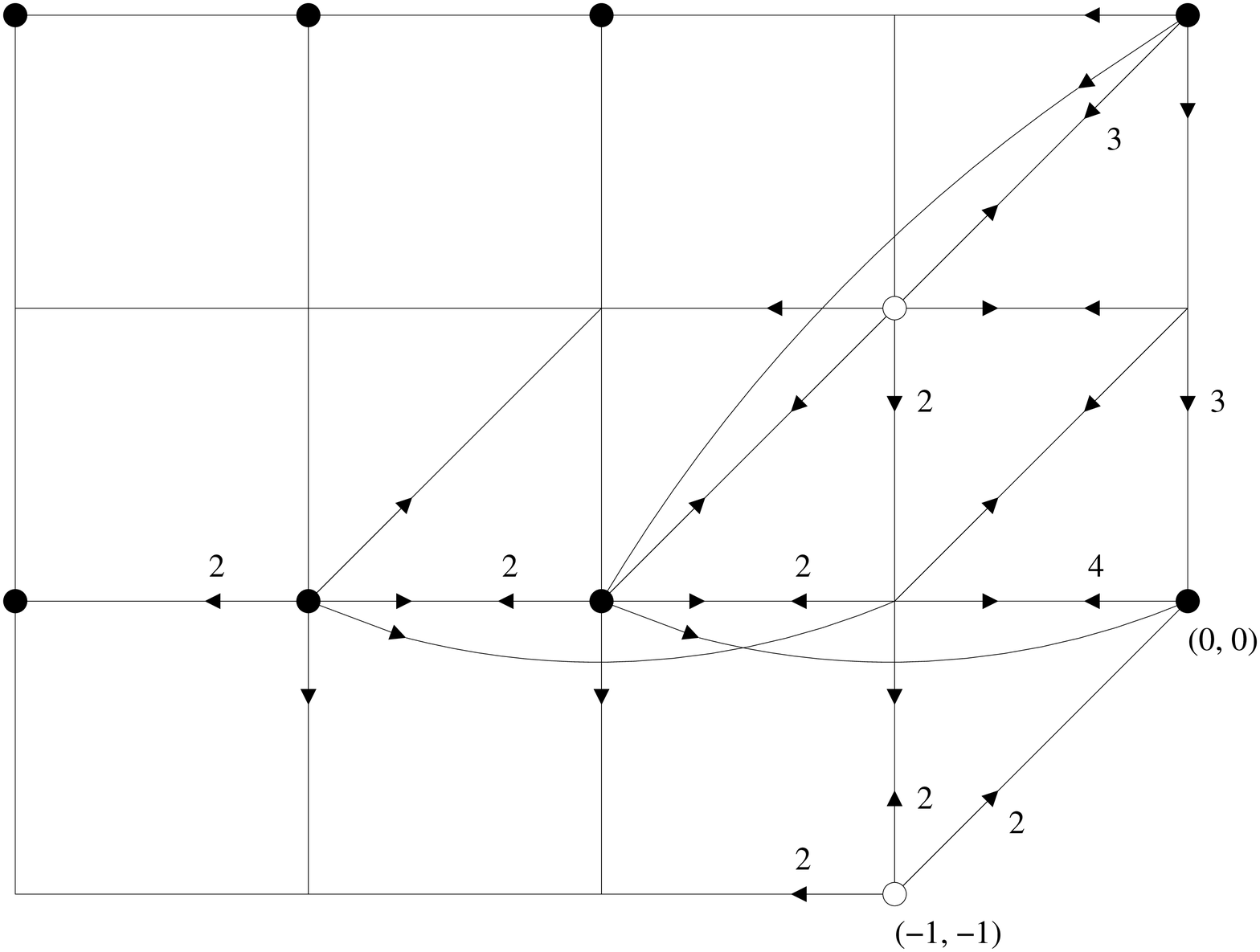}
\caption{\upshape{Transition rates of the interface}}
\label{fig:drift}
\end{figure}

\begin{lemma} --
\label{lem:deviation-1}
 There exist constants $C_1 < \infty$, $\gamma_1 > 0$ and $c > 0$ such that
 $$ P \,(\Sigma_{cN} < 7N) \ + \ P \,(\Sigma_t < - N \ \hbox{for some} \ t < cN) \ \leq \ C_1 \,\exp (-\gamma_1 N). $$
\end{lemma}
\begin{proof}
 Since the drift $D (\Sigma_t) \in \{-2, 0, 2 \}$ according to Lemma \ref{lem:drift-1}, the key step is to prove that
 the fraction of time spent on good interfaces is in average strictly larger than the fraction of time spent on bad
 interfaces, where good interfaces refer to the ones for which the sum's drift is positive and bad interfaces refer to
 the ones for which the sum's drift is negative.
 To compare these two quantities, we let $e (a, b)$ for all $(a, b) \in V$ denote the expected time spent on good interfaces
 before hitting a bad interface when starting from interface $(a, b)$, that is,
 $$ e (a, b) \ = \ E \,\bigg[\int_0^T \ind \,\{D (\Sigma_t) = 2 \} \ dt \ \Big| \ \iota (X_0) = (a, b) \bigg] $$
 where $T = \inf \,\{t > 0 : D (\Sigma_t) = -2 \}$.
 In the picture of Figure \ref{fig:drift}, good interfaces are marked with a black dot and bad interfaces with a white dot.
 Based on the transition rates given in this picture and using successive first-step analyses, we obtain that
 $$ e (0, 0) \ = \ \frac{1}{4} + e (-1, 0) \ \geq \ \frac{1}{4} + \frac{1}{5} \ e (0, 0) + \frac{2}{5} \ e (-2, 0) \ \geq \
                   \frac{1}{4} + \frac{1}{5} \times \frac{1}{4} + \frac{2}{5} \times \frac{1}{6} \ = \ \frac{11}{30} $$
 from which it follows that
 $$ \begin{array}{rcl}
     e (0, 1)  & \geq & \displaystyle \frac{3}{5} \ e (0, 0) \ \geq \ \frac{3}{5} \times \frac{11}{30} \ = \ \frac{11}{50} \vspace{8pt} \\
     e (0, 2)  & \geq & \displaystyle \frac{1}{6} + \frac{1}{6} \ e (-2, 0) + \frac{1}{6} \ \frac{3}{5} \ e (0, 0) \
                 \geq \ \displaystyle \frac{1}{6} + \frac{1}{6} \times \frac{1}{6} + \frac{1}{6} \times \frac{3}{5} \times \frac{11}{30} \ = \ \frac{52}{225} \vspace{8pt} \\
     e (-1, 0) & \geq & \displaystyle \frac{1}{5} \ e (0, 0) + \frac{2}{5} \ e (-2, 0) \
                 \geq \ \displaystyle \frac{1}{5} \times \frac{11}{30} + \frac{2}{5} \times \frac{1}{6} \ = \ \frac{7}{50} \vspace{8pt} \\
     e (-2, 0) & \geq & \displaystyle \frac{1}{6} + \frac{1}{6} \ e (0, 0) + \frac{1}{3} \ e (-3, 0) \
                 \geq \ \displaystyle \frac{1}{6} + \frac{1}{6} \times \frac{11}{30} + \frac{1}{3} \times \frac{1}{6} \ = \ \frac{17}{60}. \end{array} $$
 Using the previous lower bounds, we obtain that the expected time spent on good interfaces after leaving the bad interface $(-1, -1)$
 is bounded from below by
 $$ \frac{e (0, 0)}{3} + \frac{e (-1, 0)}{3} \ \geq \ \frac{1}{3} \times \frac{11}{30} + \frac{1}{3} \times \frac{7}{50} \ = \ \frac{38}{225} \ > \ \frac{38}{228} \ = \ \frac{1}{6} $$
 while the analog for the bad interface $(-1, 1)$ is bounded from below by
 $$ \frac{e (0, 1)}{6} + \frac{e (0, 2)}{6} + \frac{e (-1, 0)}{3} + \frac{e (-2, 0)}{6} \ \geq \ \frac{913}{5400} \ > \ \frac{913}{5478} \ = \ \frac{1}{6}. $$
 Since the expected time spent on each of the two bad interfaces at each visit is equal to $1/6$, and the previous lower bounds indicate that
 the time spent on good configurations between two consecutive visits of a bad interface is strictly larger than $1/6$, we deduce that
 $$ \liminf_{t \to \infty} \,P \,(D (\Sigma_t) = 2) \ - \ \limsup_{t \to \infty} \,P \,(D (\Sigma_t) = -2) \ = \ a \ > \ 0. $$
 In particular, we obtain the inequality
 $$ \liminf_{t \to \infty} \ t^{-1} \ \Sigma_t \ \geq \
      2 \ \liminf_{t \to \infty} \ P \,(D (\Sigma_t) = 2) \ - \ 2 \ \limsup_{t \to \infty} \ P \,(D (\Sigma_t) = -2) \ = \ 2a \ > \ 0. $$
 Let $c = 7 a^{-1}$ and $\ep = a > 0$.
 Since $\Sigma_t$ is asymptotically bounded from below by $2at$, standard large deviation estimates for the Poisson distribution imply that
 $$ \begin{array}{l}
     P \,(\Sigma_{cN} < 7N) \ + \ P \,(\Sigma_t < - N \ \hbox{for some} \ t < cN) \vspace{4pt} \\ \hspace{40pt} \leq \
     P \,(\Sigma_{cN} < (2a - \ep) cN) \ + \ P \,(\Sigma_t < - N \ \hbox{for some} \ t > 0) \vspace{4pt} \\ \hspace{40pt} \leq \
     C_2 \,\exp (-\gamma_2 N) \ + \ C_3 \,\exp (- \gamma_3 N) \end{array} $$
 for suitable constants $C_2, C_3 < \infty$ and $\gamma_2, \gamma_3 > 0$.
 This completes the proof.
\end{proof}

\begin{figure}[t!]
\centering
\includegraphics[width=400pt]{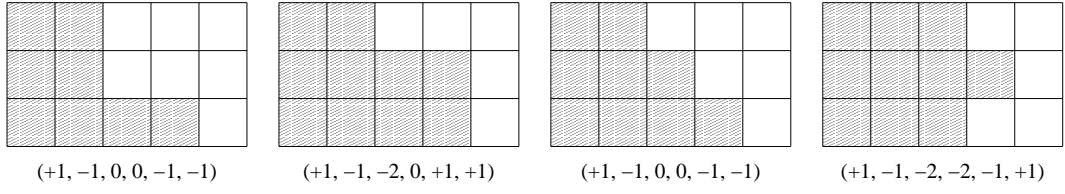}
\caption{\upshape{Pictures related to the proof of Lemma \ref{lem:drift-2}}}
\label{fig:drift-2}
\end{figure}

\begin{lemma} --
\label{lem:drift-2}
 Assume that $G_t \geq 2$.
 Then $D (G_t) \leq - 2 \times \ind \,\{X_t^+ X_t^- \neq 0 \}$.
\end{lemma}
\begin{proof}
 First, we note that, when $G_t = 2$, there are only four possible interfaces provided one identifies pairs of interfaces that can be
 deduced from one another by an axial symmetry.
 Note also that there are $G_t + 4 = 6$ possible updates for each of these four interfaces.
 Figure \ref{fig:drift-2} gives a picture of these interfaces.
 The six numbers at the bottom represent the variation of $G_t$ for each of the six possible updates.
 Since each update occurs at rate one, the drift $D (G_t)$ is simply equal to the sum of these six numbers.
 More generally, when $G_t \geq 2$, we have the following alternative.
\begin{enumerate}
 \item In the case $X_t^+ = 0$ or $X_t^- = 0$, there are only six possible updates of the interface, each of which gives the same
  variation of $D (G_t)$ as in one of the first two pictures. Therefore,
  $$ D (G_t) \ \leq \ 1 - 1 - 2 + 1 + 1 \ = \ 0 \ \ \hbox{when} \ \ X_t^+ X_t^- = 0. $$
 \item The case $X_t^+ \neq 0$ and $X_t^- \neq 0$ is similar to one of the last two pictures except that there might be one or two
  active straight points in addition to the six turns.
  Updates at these straight points cannot increase the value of the gap process, therefore,
  $$ D (G_t) \ \leq \ 1 - 1 - 1 - 1 \ = \ - 2 \ \ \hbox{when} \ \ X_t^+ X_t^- \neq 0. $$
%  \item The case $\iota (X_t) = (0, b)$, $b \geq 2$, gives the same drift as in the first picture: $D (G_t) = -2$. \vspace{2pt}
%  \item The case $\iota (X_t) = (a, 0)$, $a \leq - 2$, gives the same drift as in the second picture: $D (G_t) = 0$. \vspace{2pt}
%  \item The case $\iota (X_t) = (a, b)$, $a \leq - 1$ and $b \geq 1$, gives a drift smaller or equal than in the third picture due to
%   the possible presence of an active straight point when $a < -1$ that decreases the drift of two units: $D (G_t) \leq -2$. \vspace{2pt}
%  \item The case $\iota (X_t) = (a, b)$, $a \leq - 1$ and $b \leq - 1$, gives a drift smaller or equal than in the fourth picture
%   due to the possible presence of up to two active straight points when $a < -1$ and $b < -1$ that each decreases the drift of two
%   units: $D (G_t) \leq -4$.
\end{enumerate}
 The lemma follows.
\end{proof}

\begin{lemma} --
\label{lem:deviation-2}
 Let $c > 0$ as in Lemma \ref{lem:deviation-1}.
 Then there exist $C_4 < \infty$ and $\gamma_4 > 0$ such that
 $$ P \,(G_t > \sqrt N \ \hbox{for some} \ t < cN) \ \leq \ C_4 \,\exp (- \gamma_4 \sqrt N). $$
\end{lemma}
\begin{proof}
%  The strategy is similar to the one in the proof of Lemma \ref{lem:deviation-1}.
 For all times $s > 0$, we introduce the two stopping times
 $$ T^- (s) \ = \ \inf \,\{t > s : G_t \leq 1 \} \quad \hbox{and} \quad T^+ (s) \ = \ \inf \,\{t > s : G_t > \sqrt N \}. $$
 Since each time the slice process visits an interface such that $X_t^+ X_t^- = 0$ there is a strictly positive probability that $X_t^+ X_t^- \neq 0$
 at the next jump, we have
 $$ \liminf_{t \to \infty} \ P \,(X_t^+ X_t^- \neq 0 \ | \ G_t \geq 2) \ = \ b \ > \ 0. $$
 This, together with Lemma \ref{lem:drift-2}, gives
 $$ \liminf_{t \to \infty} \ E \,[D (G_t) \ | \ G_t \geq 2] \ \leq \ -2 \ \liminf_{t \to \infty} \,P \,(X_t^+ X_t^- \neq 0 \ | \ G_t \geq 2) \ = \ -2 b \ < \ 0 $$
 therefore, standard large deviation estimates imply that
\begin{equation}
\label{eq:deviation-1}
 P \,(T^+ (s) < T^- (s) \ | \ G_s = 2) \ \leq \ C_5 \,\exp (- \gamma_5 \sqrt N)
\end{equation}
 for suitable constants $C_5 < \infty$ and $\gamma_5 > 0$.
 Let $v_t (2)$ and $J_t$ denote respectively the number of times the gap process visits state 2 and the number of times it jumps
 by time $t$.
 Since the process jumps at rate at most 8, large deviation estimates for the Poisson distribution imply that
\begin{equation}
\label{eq:deviation-2}
 P \,(v_{cN} (2) > 5 cN) \ \leq \ P \,(J_{cN} > 10 cN) \ \leq \ C_6 \,\exp (- \gamma_6 N)
\end{equation}
 for appropriate $C_6 < \infty$ and $\gamma_6 > 0$.
 Combining \eqref{eq:deviation-1} and \eqref{eq:deviation-2}, we obtain
 $$ \begin{array}{rcl}
     P \,(G_t > \sqrt N \ \hbox{for some} \ t < cN) & \leq &
     P \,(v_{cN} (2) > 5 cN) \ + \ 5 cN \ P \,(T^+ (s) < T^- (s) \ | \ G_s = 2) \vspace{4pt} \\ & \leq &
     C_6 \,\exp (- \gamma_6 N) \ + \ 5 cN \times C_5 \,\exp (- \gamma_5 \sqrt N), \end{array} $$
 which proves the lemma.
\end{proof}

\begin{corol} --
\label{cor:drift}
 There exist $C_7 < \infty$ and $\gamma_7 > 0$ such that, for $x_2 = -1, 0, 1$,
 $$ P \,(X_{cN} (x_2) \leq 2N) \ + \ P \,(X_t (x_2) \leq -N \ \hbox{for some} \ t < cN) \ \leq \ C_7 \,\exp (- \gamma_7 \sqrt N) $$
\end{corol}
\begin{proof}
 This follows directly from the previous lemmas.
 First, since
 $$ \Sigma_t \ = \ X_t (-1) + X_t (0) + X_t (1) \ \leq \ 3 X_t (x_2) + 2 G_t  \ \ \hbox{for all $t \geq 0$ and $x_2 = -1, 0, 1$}, $$
%  $$ X_t (x_2) \ \geq \ \frac{1}{3} \ (\Sigma_t - 2 Z_t) \ \ \hbox{for all $t \geq 0$ and $x_2 = -1, 0, 1$}. $$
 a straightforward application of Lemmas \ref{lem:deviation-1} and \ref{lem:deviation-2} gives
 $$ P \,(X_{cN} (x_2) \leq 2N) \ \leq \ P \,(\Sigma_{cN} < 7N) \ + \ P \,(G_{cN} > \sqrt N) \ \leq \ C_8 \,\exp (- \gamma_8 \sqrt N) $$
 for suitable $C_8 < \infty$ and $\gamma_8 > 0$ and all $N$ sufficiently large. Similarly,
 $$ \begin{array}{l}
     P \,(X_t (x_2) \leq - N \ \hbox{for some} \ t < cN) \ \leq \
     P \,(\Sigma_t < - N \ \hbox{for some} \ t < cN) \vspace{4pt} \\ \hspace{120pt} + \
     P \,(G_t > \sqrt N \ \hbox{for some} \ t < cN) \ \leq \ C_9 \,\exp (- \gamma_9 \sqrt N) \end{array} $$
 for suitable $C_9 < \infty$ and $\gamma_9 > 0$ and all $N$ sufficiently large.
\end{proof} \\ \\
 To complete the proof of Theorem \ref{th:even-2D}, we now return to the majority rule model.
 To compare the process properly rescaled in space and time with oriented site percolation, we let $T = c N$ where $c$ is the positive constant
 introduced in Lemma \ref{lem:deviation-1}, and define for all $w \in \Z^2$
 $$ B_w \ = \ (2N + 1) \,w + [-N, N]^2 \quad \hbox{and} \quad \mathcal G \ = \ \{(w, j) \in \Z^2 \times \Z_+ : w_1 + w_2 + j \ \hbox{is even} \}. $$
 Site $(w, j) \in \mathcal G$ is said to be good whenever all vertices in $B_w$ are in state 1 at time $jT$ for the original majority rule model.
 Then, we have the following lemma.
\begin{lemma} --
\label{lem:perco}
 For all $N$ sufficiently large,
 $$ P \,((e_1, 1) \ \hbox{is not good} \ | \ (0, 0) \ \hbox{is good}) \ \leq \ C_7 \,(6N - 3) \,\exp (- \gamma_7 \sqrt N). $$
\end{lemma}
\begin{proof}
 The idea is to observe that the majority rule model is stochastically larger than a certain union of bidirectional slice processes.
 More precisely, for all integers $z \in \Z$, we let $r_t^z$ denote the process obtained from the majority rule model in the same manner as the
 slice process introduced above but applying the translation of vector $(N, z)$ to both the evolution rules and the initial configuration.
 Also, we let $l_t^z$ denote the process obtained from $r_t^z$ by applying the symmetry about the vertical axis to the evolution rules and the
 initial configuration.
 In particular, we have
\begin{equation}
\label{eq:perco-1}
 r_t^z \ \stackrel{d}{=} \ \bar \eta_t + (N, z) \quad \hbox{and} \quad l_t^z \ \stackrel{d}{=} \ - \bar \eta_t + (-N, z) \quad \hbox{for all} \ z \in \Z
\end{equation}
 where $\stackrel{d}{=}$ means equal in distribution.
 Having the majority rule model and all these processes constructed from the same collection of independent rate one Poisson processes, and
 identifying each spin system with its set of vertices in state 1, on the event that
\begin{equation}
\label{eq:perco-2}
 (- \infty, 0] \times \{z - 1, z, z + 1 \} \,\subset \,r_t^z \quad \hbox{and} \quad [0, + \infty) \times \{z - 1, z, z + 1 \} \,\subset \,l_t^z
\end{equation}
 for all $t \leq T$ and all $z \in \{- (N - 1), \ldots, N - 1 \}$, we have
\begin{equation}
\label{eq:perco-3}
 \bigcup_{z = - (N - 1)}^{N - 1} (r_T^z \cap l_T^z) \,\subset \,\eta_T \quad \hbox{provided} \quad \eta_0 \ = \bigcup_{z = - (N - 1)}^{N - 1} (r_0^z \cap l_0^z) \ = \ B_0.
\end{equation}
 Combining \eqref{eq:perco-1}-\eqref{eq:perco-3} with Corollary \ref{cor:drift}, we obtain
 $$ \begin{array}{l}
  P \,((e_1, 1) \ \hbox{is not good} \ | \ (0, 0) \ \hbox{is good}) \ \leq \
  P \,(B_{e_1} \not \subset \eta_T \ | \ \eta_0 = B_0) \vspace{4pt} \\ \hspace{25pt} \leq \
  P \,((- \infty, 0] \times \{z - 1, z, z + 1 \} \not \subset r_t^z \ \hbox{for some} \ (z, t) \in \{- (N - 1), \ldots, N - 1 \} \times (0, T)) \vspace{4pt} \\ \hspace{50pt} + \
  P \,((- \infty, 3N] \times \{z - 1, z, z + 1 \} \not \subset r_T^z \ \hbox{for some} \ z \in \{- (N - 1), \ldots, N - 1 \}) \vspace{4pt} \\ \hspace{25pt} \leq \ (2N - 1) \times
  P \,((- \infty, -N] \times \{- 1, 0, 1 \} \not \subset \bar \eta_t \ \hbox{for some} \ t \in (0, T)) \vspace{4pt} \\ \hspace{50pt} + \ (2N - 1) \times
  P \,((- \infty, 2N] \times \{- 1, 0, 1 \} \not \subset \bar \eta_{cN}) \ \leq \ C_7 \,(6N - 3) \,\exp (- \gamma_7 \sqrt N) \end{array} $$
 for all $N$ sufficiently large, as desired.
\end{proof} \\ \\
 Since the probability in the statement of Lemma \ref{lem:perco} can be made arbitrarily small by choosing the parameter $N$
 sufficiently large, Theorem 4.3 in Durrett \cite{durrett_1995} implies that the set of good sites dominates stochastically the set of
 wet sites of an oriented site percolation process on $\mathcal G$ where sites are open with probability arbitrarily close to one.
 This only proves survival of the type 1 opinion since the percolation process has a positive density of closed sites, and thus a
 positive density of dry sites, i.e., sites which are not wet.
 To conclude, we apply Lemma 15 of \cite{lanchier_2012b}, which relies on ideas from Durrett \cite{durrett_1992} and proves the lack
 of percolation of the dry sites for oriented percolation on a certain directed graph with vertex set $\mathcal G$ when the density
 of open sites is large enough.
 This lemma and the construction given in its proof imply the existence of an in-all-directions expanding region which is void of
 vertices in state 0, so opinion 1 indeed outcompetes opinion 0 when starting with infinitely many $2 \times 2$ squares in state 1.
 This completes the proof of Theorem \ref{th:even-2D}.

%%%%%%%%%%%%%%%%%%%%%%%%%%%%%%%%%%%%%%%%%%%%%%%%%%%%%%%%%%%%%%%%%%%%%%%%%%%%%%%%%%%%%%%%%%%%%%%%%%%%%%%%%%%%%%%%%%%%%%%%%%%%%%%%%%%%%%%%%%

\section{Proof of Theorem \ref{th:odd-2D} ($d = 2$ and $n = 3$)}
\label{sec:odd-2D}

\indent This last section is devoted to the proof of Theorem \ref{th:odd-2D}, which relates the variation rate of the number
 of type 1 vertices to the number of positive and negative corners when the configuration is a regular cluster.
 Recall that configuration $\eta$ is a regular cluster whenever
\begin{enumerate}
 \item[(R0)] the contour $\Gamma$ as defined in Section \ref{sec:results} is a Jordan curve, \vspace{4pt}
 \item[(R1)] the set $(x + D_2) \cap \Gamma$ is connected for all $x \in \D^2$ and \vspace{4pt}
 \item[(R2)] if vertex $x$ is a corner and $(x + D_2) \cap (y + D_2) \neq \varnothing$ then vertex $y$ is not a corner,
\end{enumerate}
 where $D_2 = [- 1, 1]^2$.
 The proof is divided into three steps:
 we first prove a geometric property of regular clusters, then establish the theorem in the particular case when there
 is no positive corner nor negative corner, and finally combine these two results to obtain the full theorem.

\begin{figure}[t!]
\centering
\includegraphics[width=180pt]{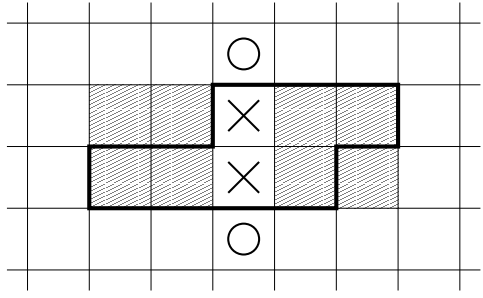}
\caption{\upshape{Picture related to the proof of Lemma \ref{lem:tictactoe}}}
\label{fig:tictactoe}
\end{figure}

\begin{lemma} --
\label{lem:tictactoe}
 Assume that $\eta$ is a regular cluster with at least 11 vertices.
 If two nearest neighbors, say vertex $x$ and vertex $x + e_i$ for some $i \in \{1, 2 \}$, are in different states then
 $$ \eta (x - 2 e_i) = \eta (x - e_i) = \eta (x) \neq \eta (x + e_i) = \eta (x + 2 e_i) = \eta (x + 3 e_i). $$
\end{lemma}
\begin{proof}
 Accounting for the invariance of the problem by translation and rotation, it suffices to prove that each of the following four
 scenarios leads to a contradiction:
\begin{enumerate}
 \item $\eta (e_2) \neq \eta (0) = \eta (2e_2) = 1$ \vspace{4pt}
 \item $\eta (e_2) = \eta (2e_2) \neq \eta (0) = \eta (3e_2) = 1$ \vspace{4pt}
 \item $\eta (e_2) \neq \eta (0) = \eta (2e_2) = 0$ \vspace{4pt}
 \item $\eta (e_2) = \eta (2e_2) \neq \eta (0) = \eta (3e_2) = 0$
\end{enumerate}
 We focus on conditions 2 and 4 since the remaining two conditions can be excluded based on the same approach.
 Using the notations of the tic-tac-toe game by denoting each state by $\times$ or $\circ$ where $\times$ means either state 0
 or state 1 and $\circ$ means the other state, both conditions 2 and 4 result in the configuration of $\times$ and $\circ$
 given in Figure \ref{fig:tictactoe}, which implies the existence of two horizontal line segments of length one that must be
 subsets of the contour.
 According to (R1), the hatched square on the left of the picture must contain a path $\gamma_1$ that connects the left
 extremities of the two segments.
 Similarly, the right extremities must be connected by a path $\gamma_2$ contained in the hatched square on the right of the
 picture.
 The concatenation of the two segments and the two paths defines a Jordan curve $\gamma$.
 To conclude, we distinguish between the following two conditions.
\begin{enumerate}
 \item If $\times$ means state 1 then $\gamma \subsetneq \Gamma$ contradicting the fact that $\Gamma$ is a Jordan curve. \vspace{4pt}
 \item If $\times$ means state 0 then $\gamma = \Gamma$ indicating that the cluster contains at most 10 vertices, which again
  leads to a contradiction.
\end{enumerate}
 This completes the proof.
\end{proof}

\begin{figure}[t!]
\centering
\includegraphics[width=400pt]{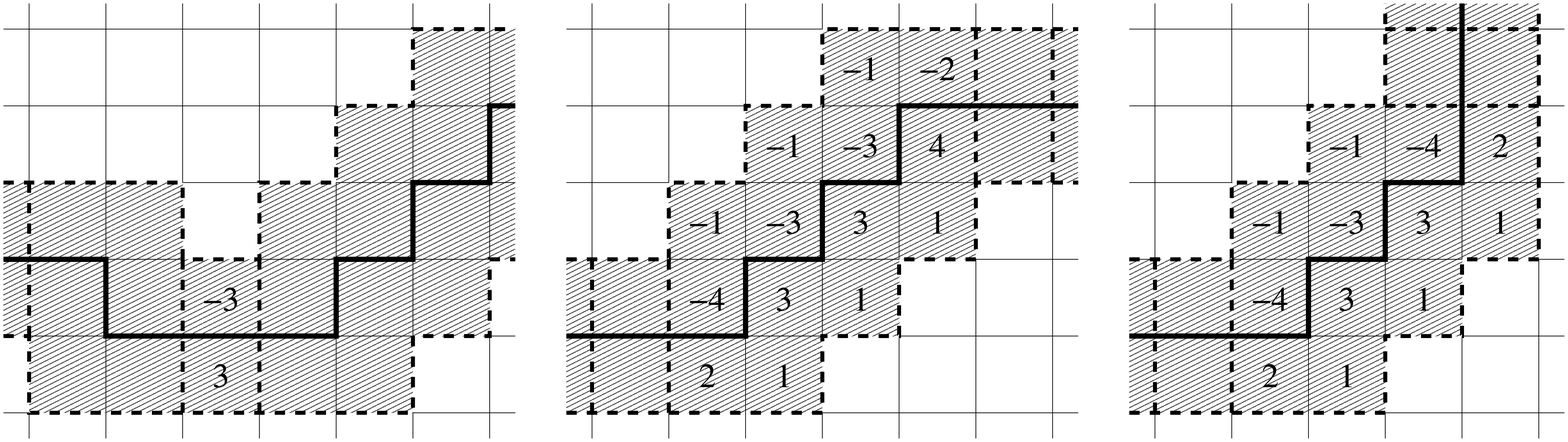}
\caption{\upshape{Pictures related to the proof of Lemma \ref{lem:neutral}}}
\label{fig:neutral}
\end{figure}

\begin{lemma} --
\label{lem:neutral}
 Theorem \ref{th:odd-2D} holds when $c_- = c_+ = 0$.
\end{lemma}
\begin{proof}
 The proof relies on a geometric construction much easier to visualize than to explain so we refer the reader to Figure \ref{fig:neutral}
 for pictures that help to understand our approach.
 The basic idea is to define a partition $\{\Delta_1, \Delta_2, \ldots, \Delta_k \}$ of the support of $\phi (\eta, \cdot \,)$ such that
 the property to be proved holds for each member of the partition, i.e.,
\begin{equation}
\label{eq:neutral}
 \sum_{x \in \Delta_j} \ \phi (\eta, x) \ = \ 0 \quad \hbox{for all} \ j = 1, 2, \ldots, k.
\end{equation}
 First, we define an oriented contour embedded in the Jordan curve $\Gamma$ by letting
 $$ \{x_1, x_2, \ldots, x_m \} \ = \ \Gamma \,\cap \,\D^2 $$
 where $x_i$ is the $i$th point we encounter going around the curve starting from a given point and following a given orientation.
 To turn the Jordan curve into an oriented contour, we draw an arrow from vertex $x_i$ to vertex $x_j$ whenever $j = i + 1 \mod m$, in
 which case both vertices are nearest neighbors on the dual lattice.
 Then, we define
 $$ \Lambda \ = \ \bigcup_{i = 1}^m \ (x_i + D_2) \qquad \hbox{and} \qquad \Delta \ = \ \supp \,\phi (\eta, \cdot \,) \ = \ \Lambda \,\cap \,\Z^2. $$
 If the arrows $(x_{i - 1}, x_i)$ and $(x_i, x_{i + 1})$ are oriented in the same direction, we draw a segment line of length two
 centered at $x_i$ and perpendicular to the segment $(x_{i - 1}, x_{i + 1})$.
 This induces a partition of the set $\Lambda$ and a partition of the support $\Delta$, i.e.,
 $$ \Lambda \ = \ \bigcup_{j = 1}^k \ \Lambda_j \qquad \hbox{and} \qquad
    \Delta  \ = \ \bigcup_{j = 1}^k \ \Delta_j \ = \ \bigcup_{j = 1}^k \ (\Lambda_j \cap \Z^2) $$
 where the unions are disjoint.
 In Figure \ref{fig:neutral}, the sets $\Lambda_j$ are delimited by dashed lines.
 Now, thinking of the contour as a sequence of length $m$ consisting of four different types of arrows, a direct application
 of Lemma \ref{lem:tictactoe} implies that two consecutive vertical arrows going in opposite direction must be separated by at least
 three horizontal arrows all going in the same direction.
 The same holds by exchanging the role of vertical and horizontal arrows.
 Each sequence of $l \geq 3$ consecutive arrows oriented in the same direction induces $l - 2$ members of the partition of
 the support with exactly two vertices such that
 $$ \Delta_j \ = \ \{x, y \} \quad \hbox{with} \quad \phi (\eta, x) \ + \ \phi (\eta, y) \ = \ 3 \ - \ 3 \ = \ 0. $$
 See the first picture of Figure \ref{fig:neutral} for an illustration.
 Finally, the absence of positive and negative corners implies that the path $(\rightarrow, \rightarrow, \uparrow, \uparrow)$
 as well as the seven other paths deduced by symmetry or rotation are not allowed.
 In particular, accounting again for symmetry and rotation, for all $j$ such that $\card \Delta_j \neq 2$, we must have
 $$ \Lambda_j \,\cap \,\Gamma \ = \ (\rightarrow, \uparrow, \rightarrow, \uparrow, \cdots, \rightarrow, \uparrow) \qquad \hbox{or} \qquad
    \Lambda_j \,\cap \,\Gamma \ = \ (\rightarrow, \uparrow, \rightarrow, \uparrow, \cdots, \rightarrow, \uparrow, \rightarrow) $$
 where both paths have length at least three, both paths must be preceded by a $\rightarrow$, the first path must be followed by a $\uparrow$
 and the second path must be followed by a $\rightarrow$.
 Such paths are represented in the right hand side of Figure \ref{fig:neutral} which gives the values of $\phi (\eta, \cdot \,)$ in the
 set $\Delta_j$ and shows that \eqref{eq:neutral} is indeed satisfied.
 Since the sets $\Delta_j$ form a partition, the proof is complete.
\end{proof}

\begin{figure}[t!]
\centering
\includegraphics[width=280pt]{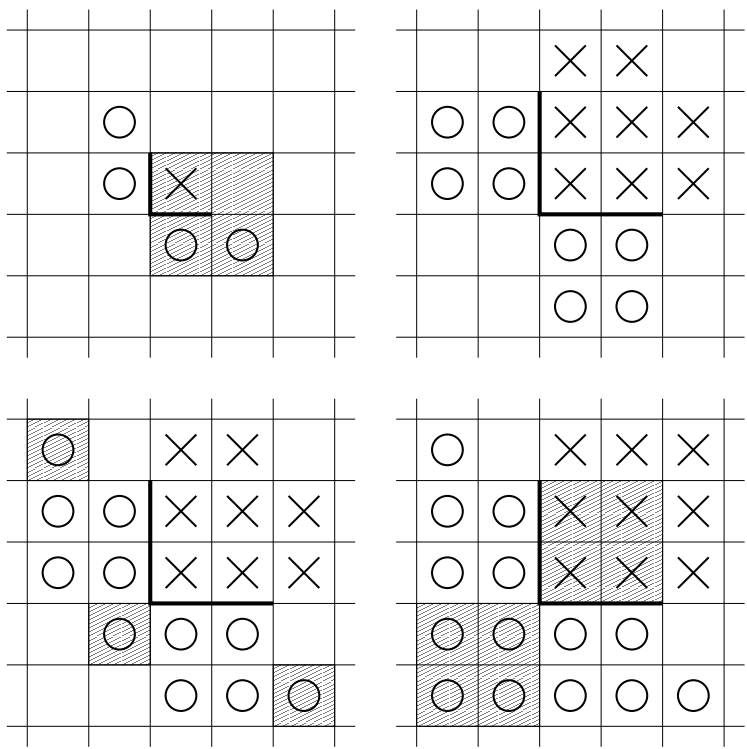}
\caption{\upshape{Pictures related to the proof of Lemma \ref{lem:corner}}}
\label{fig:corner}
\end{figure}

\begin{lemma} --
\label{lem:corner}
 Theorem \ref{th:odd-2D} holds for all values of $c_-$ and $c_+$.
\end{lemma}
\begin{proof}
 The first step is to characterize the configuration of the process in the $5 \times 5$ square centered at a positive or negative
 corner, which is illustrated in Figure \ref{fig:corner}.
 Following the notations introduced above, we denote both states by $\times$ and $\circ$ respectively.
 Since the problem is invariant by translation and rotation, we may assume without loss of generality that vertex 0 is a corner with
 $$ \eta (0) \ \neq \ \eta (- e_1 + e_2) \ = \ \eta (e_1 - e_2) $$
 as illustrated by the top left picture of the figure.
 We now prove that the configuration constructed step by step in the figure is indeed the only possible configuration.
 \begin{list}{\labelitemi}{\leftmargin=1em}
 \item[] {\bf Top left picture} --
  The top left corner and the bottom right corner of the unit square centered at vertex 0 must belong to the Jordan curve $\Gamma$.
  Invoking condition (R1), we deduce the existence of a path connecting these two points and included in the hatched square of the
  picture, which further implies that the Jordan curve must contain the top side or the left side of the unit square centered at 0.
  By symmetry, it must also contain the bottom side or the right side.
  A direct application of Lemma \ref{lem:tictactoe} shows that two opposite sides of the unit square cannot simultaneously be included
  in the curve $\Gamma$ so we may assume without loss of generality that the left and bottom sides of the unit square, drawn in thick
  lines in the picture, are included in the Jordan curve, which also determines the state of the vertices to the left of and under
  vertex 0. \vspace{4pt}
 \item[] {\bf Top right picture} --
  Applying Lemma \ref{lem:tictactoe} repeatedly from the previous picture gives the state of a total of 16 vertices included in the
  $5 \times 5$ square, as shown in this picture. \vspace{4pt}
 \item[] {\bf Bottom left picture} --
  Invoking condition (R2), the three hatched unit squares in this picture cannot be positive or negative corners, which forces
  their state to be $\circ$ rather than $\times$. \vspace{4pt}
 \item[] {\bf Bottom right picture} --
  The top right corner of the $2 \times 2$ hatched square on the left is a point of the Jordan curve $\Gamma$ therefore the presence
  of a $\times$ in this hatched square would contradict condition (R1).
  It follows that all four vertices in this square are $\circ$.
  Using a similar reasoning with the other hatched square, we prove that the vertex at the top right of the picture must be $\times$,
  while the remaining two vertices with no symbol can be of either type.
\end{list}
 To deduce the theorem, we let $\bar \eta$ denote the configuration obtained from $\eta$ by switching the state of vertex 0 and
 leaving the state of all other vertices unchanged.
 It is straightforward to check that, regardless of the type of the remaining two vertices, none of the nine $3 \times 3$ squares
 containing vertex 0 contains exactly five $\times$'s, indicating that the majority type in each of these nine square is not
 modified by switching the state of 0.
 In particular, each of the nine corresponding updates results in the same configuration that we start from $\eta$ or from
 $\bar \eta$ which further implies that
\begin{equation}
\label{eq:corner}
  \phi (\eta, x) \ = \ \left\{\hspace{-3pt}
 \begin{array}{rcl}
  \phi (\bar \eta, x) + 1 & \hbox{if} & \eta (0) = 0 \vspace{3pt} \\
  \phi (\bar \eta, x) - 1 & \hbox{if} & \eta (0) = 1
 \end{array} \right.
\end{equation}
 for every vertex $x$ in the $3 \times 3$ square centered at 0.
 More generally, letting $\bar \eta$ be the configuration obtained from $\eta$ by switching the state of all corners, we get
 $$ \sum_{x \in \Z^2} \ \phi (\eta, x) \ = \ \sum_{x \in \Z^2} \ \phi (\bar \eta, x) \ + \ 9 \,(c_- - c_+) \ = \ 9 \,(c_- - c_+) $$
 where the first equation is obtained by applying \eqref{eq:corner} at each corner, and where the second equation follows from
 Lemma \ref{lem:neutral} and the fact that the configuration $\bar \eta$ obtained by switching the state of each corner is a
 regular cluster with no corner.
 This completes the proof.
\end{proof} \\

%%%%%%%%%%%%%%%%%%%%%%%%%%%%%%%%%%%%%%%%%%%%%%%%%%%%%%%%%%%%%%%%%%%%%%%%%%%%%%%%%%%%%%%%%%%%%%%%%%%%%%%%%%%%%%%%%%%%%%%%%%%%%%%%%%%%%%%%%%

\noindent\textbf{Acknowledgment}.
 The authors would like to thank two anonymous referees for many comments that helped to improve the clarity of this article.

%%%%%%%%%%%%%%%%%%%%%%%%%%%%%%%%%%%%%%%%%%%%%%%%%%%%%%%%%%%%%%%%%%%%%%%%%%%%%%%%%%%%%%%%%%%%%%%%%%%%%%%%%%%%%%%%%%%%%%%%%%%%%%%%%%%%%%%%%%

\end{document}